\documentclass[a4paper]{article}

\usepackage{amsmath,amssymb,mathtools,amsthm}
\usepackage{graphicx,float,enumerate}
\graphicspath{{images/}}

\usepackage[default,scale=1]{opensans}
\usepackage[T1]{fontenc}

\usepackage[utf8]{inputenc}
\DeclareUnicodeCharacter{00A0}{ }
\usepackage{geometry}
\newcommand{\ee}{\varepsilon}

\newcommand{\BUC}{{\mathrm{BUC}}}
\DeclareMathOperator{\N}{N}

\newcommand{\Rd}{{\mathbb R^d}}

\newcommand{\diver}{\nabla \cdot}

\DeclareMathOperator{\supp}{supp}

\newcommand{\vv}[1]{\mathbf{#1}}

\newcommand*\diff{\mathop{}\!\mathrm{d}}

\usepackage[hidelinks]{hyperref}
\usepackage[nameinlink]{cleveref}

\newtheorem{theorem}{Theorem}[section]
\newtheorem{proposition}[theorem]{Proposition}%
\newtheorem{corollary}[theorem]{Corollary}%
\newtheorem{lemma}[theorem]{Lemma}%

\theoremstyle{definition}
\newtheorem{definition}[theorem]{Definition}%
\newtheorem{remark}[theorem]{Remark}%

\numberwithin{equation}{section}

\usepackage[
url=false,
isbn=false,
backend=bibtex,
maxcitenames=4,
giveninits,
maxbibnames=100,
doi=true,
uniquename=false,
block=none]{biblatex}
\renewbibmacro{in:}{}
\DeclareFieldFormat[article]{citetitle}{#1}
\DeclareFieldFormat[article]{title}{#1}

\makeatletter
\renewcommand*{\@fnsymbol}[1]{\ensuremath{\ifcase#1\or \star \or \dagger\or \ddagger\or
		\mathsection\or \mathparagraph\or \|\or **\or \dagger\dagger
		\or \ddagger\ddagger \else\@ctrerr\fi}}
\makeatother

\begin{document}
\title{A fast regularisation of a Newtonian vortex equation}

\author{José A. Carrillo\thanks{Mathematical Institute, University of Oxford, Oxford OX2 6GG, UK. Email: \href{mailto:carrillo@maths.ox.ac.uk}{carrillo@maths.ox.ac.uk}} %
\and David Gómez-Castro\thanks{Mathematical Institute, University of Oxford, Oxford OX2 6GG, UK. Email: \href{mailto:gomezcastro@maths.ox.ac.uk}{gomezcastro@maths.ox.ac.uk}} \thanks{Instituto de Matemática Interdisciplinar, Universidad Complutense de Madrid, Spain.} \and %
Juan Luis Vázquez\thanks{Departamento de Matemáticas, Universidad Autónoma de Madrid, Spain. Email: \href{mailto:juanluis.vazquez@uam.es}{juanluis.vazquez@uam.es}}}
\maketitle

\begin{abstract}
We consider equations of the form $	u_t = \diver ( \gamma(u)  \nabla \N(u))$,
where $\N$ is the Newtonian potential (inverse of the Laplacian) posed in the whole space $\Rd$, and
$\gamma(u)$ is the mobility. For linear mobility, $\gamma(u)=u$, the equation and some variations have been proposed as a model for superconductivity or superfluidity. In that case the theory  leads to uniqueness of bounded weak solutions having the property of compact space support, and in particular there is a special solution in the form of a disk vortex of constant intensity in space $u=c_1t^{-1}$ supported in a ball that spreads in time like $c_2t^{1/d}$, thus showing a discontinuous leading front.

 In this paper we propose the model with sublinear mobility $\gamma(u)=u^\alpha$, with $0<\alpha<1$, and prove that nonnegative solutions recover positivity everywhere, and moreover display a fat tail at infinity. The model acts in many ways as a regularization of the previous one. In particular, we find that the equivalent of the previous vortex is an explicit self-similar solution decaying in time like $u=O(t^{-1/\alpha})$ with a space tail with size $u=O(|x|^{- d/(1-\alpha)})$. We restrict the analysis to radial solutions and construct solutions by the method of characteristics. We introduce the mass function, which solves an unusual variation of Burger's equation, and plays an important role in the analysis. We show well-posedness in the sense of viscosity solutions. We also construct numerical finite-difference convergent schemes.
\end{abstract}

\


\smallskip


\setcounter{tocdepth}{2}

\

\section{Introduction}
\label{sec.intro}

We will study equations on the form
\begin{equation}\label{oureq}
		u_t = \diver ( \gamma(u)  \nabla \N(u))
\end{equation}
where $\N$ is the Newtonian potential
\begin{equation*}
	\N(u(t,\cdot)) = \int_{\mathbb R^d} \mathbb G (x,y) u(t,y) \diff y
\end{equation*}
for $\mathbb G$ the Green kernel, and $\gamma(u)$ is called the mobility. For linear mobility $\gamma(u)=u$,
the equation has  been studied by a number of authors as a model for superconductivity or superfluidity, cf. Lin and Zhang \cite{Lin+Zhang2000}, Ambrosio, Mainini, and Serfaty \cite{Ambrosio+Mainini+Serfaty2011,Ambrosio+Serfaty2008}, Bertozzi, Laurent, and L\'eger \cite{BLL12}, Serfaty and Vazquez \cite{Serfaty+Vazquez2014}. The theory of the last paper leads to uniqueness of bounded weak solutions having the property of compact support, and in particular to a special solution in the form of a disk vortex of constant intensity in space  that decays in time like
$c_1\,t^{-1}$ and is supported in a ball that spreads with radius $R=c_2 t^{1/d}$, thus showing a discontinuous leading front, i.e.\ $u=c_1t^{-1} \chi_{B(0, c_2 t^{1 / d})}$.
This vortex solution is an asymptotic attractor for a large class of solutions. Moreover, in dimension 2, the equation is directly related to the Chapman-Rubinstein-Schatzman \cite{Chapman1996} mean field model of superconductivity and to E's model of superfluidity \cite{E1994}, which would correspond rather to the equation $u_t = \diver (|u|\nabla p).$

On the other hand, we can formally understand \eqref{oureq} as a gradient flow equation with the nonlinear mobility $\gamma(u)$ by rewriting it as
$$
u_t = \diver \left( \gamma(u)  \nabla \frac{\delta \mathcal{F}}{\delta u}\right) \qquad \mbox{with }  \frac{\delta \mathcal{F}}{\delta u}=\N(u),
$$
and the associated energy functional
$$
\mathcal{F}(u)=\frac12 \int_{\mathbb R^d} \N(u) u \diff x\, .
$$
The transport distance associated to this nonlinear continuity equation was shown in \cite{DNS09} to be well defined for nonlinear mobilities of the form $\gamma(u)=u^\alpha$, $0<\alpha <1$, and for general concave nonlinear mobilities, while transport distances associated with convex nonlinear mobilities are not well-defined in general. Gradient flows associated to homogeneous concave mobilities were studied subsequently in \cite{CLSS10}. This interesting line of research will not be pursued further in this paper.

\paragraph{Statement of the problem and outline of results.}
In this paper we study the problem with nonlinear mobility $\gamma(u)=u^\alpha$,  with $0<\alpha <1$.
The presence of the sublinear mobility leads to a number of results that strongly depart from the linear mobility case, and at the time implies the need for significant new tools to develop the theory. In particular, we show that the sublinear nonlinearity eliminates the compact support effect of the typical vortex solutions, and leads to profiles with fat tails at infinity (of the space variable). They can be interpreted as a diffused vortex. Moreover, the tails depend on a very precise way of the exponent $\alpha<1$.
The case $\alpha \ge 1$ leads to a completely different behaviour: compactly supported self-similar solutions (see \cite{Carrillo+GC+Vazquez2020}).
We write the problem
 \begin{equation}
	\tag{P}
	\label{eq:main equation}
	\begin{dcases}
		u_t = \diver ( u^\alpha   \nabla v) & (0,+\infty) \times \mathbb R^d \\
		-\Delta v = u & (0,+\infty) \times \mathbb R^d \\
		u(t,x), v(t,x) \to 0 & |x| \to +\infty \\
		u = u_0 & t = 0.
	\end{dcases}
\end{equation}
in all space dimensions $d\ge 1$.
We assume that $u_0 \ge 0$. We will show that this implies that $u \ge 0$.

In the first part of this paper, Sections \ref{sec:explicit solutions}-\ref{sec.gendata}, we will focus on constructing radial weak solutions by characteristics, introducing rarefaction fans and shocks as appropriate. This will sometimes lead to the existence of multiple weak solutions for certain initial data.
The second part, Sections \ref{sec.exist}-\ref{sec.fds}, deals with the selection of the stable solutions in the sense of vanishing viscosity and the notion of viscosity solution of the mass equation present below. This allows for a well-posedness theory of the equation \eqref{eq:main equation} for radial solutions. We now explain in detail the main results of each section.

We begin our study in Section \ref{sec:explicit solutions} by looking for relevant  explicit solutions. Notably, we find a selfsimilar solution with finite mass that will be the equivalent in our model of the vortex solution
mentioned above for linear mobility. This solution is explicit, radially symmetric, and it has power decay rate in space for every $t>0$ while it decays like $O(t^{-1/\alpha})$ in sup norm.
In particular, we will show that the self-similar solution of total mass $M$ is given by
\begin{equation}
	\label{eq:self-sim solution of mass M}
	U_M(t,x) = t^{-\frac 1 \alpha} \left( \alpha +
	\left( \frac{\omega_d |x|^d t^{-\frac 1 {\alpha}} } { \alpha M} \right)^{\frac {\alpha} {1-\alpha }}
	\right)^{-1/\alpha }.
\end{equation}
Letting $\alpha\to 1$ we get the compactly supported vortex created by the equation with linear mobility.

In the first part of the paper we are particularly interested in radial solutions for which a very detailed description can be obtained. For these solutions we can study the mass function, which is introduced in Section \ref{sec.mass} as
\begin{equation}
	m(t,r) = \int_{B_r} u(t,x) dx
\end{equation}
which is the solution of a Hamilton-Jacobi type equation when written in the volume variable $\rho = \omega_d r^d$:
\begin{equation}\label{eq.mrho}
	m_t + m (m_\rho)^{\alpha}  = 0.
\end{equation}
This equation has a reminiscence to Burger's equation. Indeed, it is a very unusual version of it that needs a careful development. To remark that our study is dimension independent. We recall that $m_\rho=u\ge0$.

Equation \eqref{eq.mrho} will be studied by the method of characteristics, following \cite{Evans1998}. This is done in Subsection \ref{ssect.gch} and we obtain solutions by gluing characteristic lines
(see \Cref{thm:characteristics})
.
In particular, we recover
 again
the selfsimilar solution
\eqref{eq:self-sim solution of mass M}.
We devote \Cref{sec.radial}
to show that the method of characteristics
works well when $u_0(r)$ is radially symmetric and decreasing.
First, in \Cref{ssec:characteristics decreasing continuous} we discuss the case where $u_0$ is non-increasing and continuous, and the characteristics fill the space. Then, in \Cref{ssec:characteristics decreasing discontinuous} we study the case in which $u_0$ is non-increasing and discontinuous, where characteristics leave gaps. One way to fill these gaps is the introduction of a rarefaction fan, which is presented in \Cref{ssect.rf}.
This important topic is treated in detail. Then we derive mass conservation
(\Cref{prop:characteristics mass conservation}),
comparison principle
(\Cref{thm:comparison principle characteristics}),
and asymptotic behaviour for such solutions
(as $t \to \infty$ in
\Cref{thm:asymptotics t non-increasing u_0,lem:relative error with self-similar power alpha} and for fixed $t$ as $|x| \to \infty$ in \Cref{thm:non-decreasing analysis of the tails}).

Next, we enlarge the class of initial data in Section \ref{sec.gendata}, still radially symmetric,
but only piece-wise decreasing.
Then shocks may appear, and we need  Rankine-Hugoniot conditions
(given by \eqref{eq:RH})
to select the correct shock solutions.
In fact, we give in \Cref{sec:spurious}
an example of non-uniqueness of weak solutions: the square functions.

We then address the issue of constructing solutions for a large class of initial data and selecting the physical ones. We devote two sections to construct viscosity approximations, as is customary to do for similar problems. In \Cref{sec.exist} we will consider a regularised problem with a viscous term $\ee \Delta u$:
\begin{equation}
	\tag{P$_\varepsilon$}
	\label{eq:PDE regularised}
	\begin{dcases}
		u_t = \diver ( (\ee + u_+)^\alpha \nabla v) + \ee \Delta u & (0,+\infty)  \times \mathbb R^d \\
		-\Delta v = u & (0,+\infty) \times \mathbb R^d \\
		u(t,x), v(t,x) \to 0 & |x| \to +\infty \\
		u = u_0 & t = 0.
	\end{dcases}
\end{equation}
The limit of this problem as $\ee \to 0$ is called the vanishing viscosity limit.
We prove that, for general (non-radial) initial data, \eqref{eq:PDE regularised} is well-posed (\Cref{thm:vanishing viscosity well-posedness}), has suitable $L^p$ estimates (\Cref{prop:PDE regularised estimates}), its mass satisfies  \eqref{eq:PDE regularised mass} similar to \eqref{eq.mrho} and it converges in the sense of weak solutions (\Cref{thm:vanishing viscosity ee to 0}).
Passing to the limit $\ee\to 0$ thanks to suitable a priori estimates we get weak solutions for quite general,
not necessarily radial data.

We still have the problem of uniqueness that we solve for radially symmetric data by passing to the limit in the above approximation, but
now in the mass variable.
In \Cref{sec.viscos} we obtain
 obtain a unique viscosity solution in the sense of Crandall-Lions, \cite{Crandall1983}.
We prove that bounded and uniformly continuous viscosity solutions of \eqref{eq.mrho} satisfy a comparison principle (\Cref{thm:comparison principle m}) and can be recovered as the limit of the solutions of \eqref{eq:PDE regularised mass} (\Cref{thm:converges mass PDE regularised}). This allows us to state the well-posedness in \Cref{thm:viscosity mass well-posedness}. We conclude the section by discussing the asymptotic behaviour of viscosity solutions in \Cref{thm:asymptotics mass}.

Finally, we devote Section \ref{sec.fds} to construct numerical finite-difference convergent schemes for the mass function
using viscosity-solution techniques.
Numerical calculations illustrate the main results of the paper at different stages. We close the paper with some comments on extensions and open problems in
\Cref{sec.extop}.

\section{Explicit solutions}
\label{sec:explicit solutions}

In this section we construct two families of explicit  solutions for \eqref{eq:main equation}.

\subsection{Constant in space solutions and Friendly Giant}

We look for ODE type solutions for \eqref{eq:main equation}. Indeed, for initial  constant data $u_0(x)$ we may look for supersolutions
$
	 u(t,x) = g(t)
$.
Writing the equation
\begin{equation*}
	u_t = \diver( u^\alpha \nabla v ) = \nabla u^\alpha \nabla v + u^\alpha \Delta v = \nabla u^\alpha \nabla v - u^{\alpha+1}.
\end{equation*}
Hence,
\begin{equation*}
	 g' = - g^{\alpha+1}.
\end{equation*}
Therefore, we have the friendly giant solution:
\begin{equation*}
	 g (t) = (u_0^{-\alpha} + \alpha t)^{-1/\alpha}.
\end{equation*}
Assuming that a comparison principle works, this solution will allow below to show that
\begin{equation}
	\overline u (t,x) = (\|u_0\|_{L^\infty}^{-\alpha} + \alpha t)^{-1/\alpha}
\end{equation}
is a supersolution.

\paragraph{Global supersolution}
Even, as $\| u_0 \|_{L^\infty} \to +\infty$ we have the so-called Friendly Giant \begin{equation}
	{\widetilde u} (t) = (\alpha t)^{-1/\alpha}.
\end{equation}
Even if these solutions are not in $L^1$, comparison works for any viscosity solution or for any limit of approximate classical solutions like the ones obtained by the vanishing viscosity method.

\subsection{Self-similar solutions}
\label{sec:selfsimilar via equation}
Next, we establish the existence of the important class of selfsimilar solutions, which take the form
\begin{equation}
	\label{eq:solution is self-similar form}
	U(t,x)= t^{-\gamma} F(|x|\,t^{-\beta}).
\end{equation}
In order to satisfy \eqref{eq:main equation} and conserve mass we must take
$$
\gamma=\frac1{\alpha }, \quad \beta=\frac1{\alpha d}.
$$

\paragraph{A PDE in self-similar variables.} Then the equation for the profile $U(t,x) = t^{-\gamma} F(|y|)$ where $y = x t^{-{\beta}}$ is
$$
-\frac 1{\alpha d} \nabla \cdot (y\,F )=\nabla\cdot(F ^\alpha \nabla \N(F )).
$$
Eliminating the nablas and rearranging, we get the fractional stationary equation
$$
yF ^{1-\alpha }=-\alpha d \nabla \N(F ).
$$
Applying the divergence operator to the latter equation, we get
\begin{equation}\label{eq.sss1}
\nabla\cdot (yF ^{1-\alpha })=-\alpha d \Delta \N(F )= \alpha d \,F
\end{equation}
since $\N$ is the inverse of $-\Delta$ in $\mathbb R^d$.

\paragraph{An ODE for $F$ in radial coordinates.} In order to solve this equation we put $w(r)=F(r)^{1-\alpha }$ so that $F=w^p$ with $p=1/(1-\alpha )>1$. Notice that $p-1=\alpha /(1-\alpha )$. Also, $p\to \infty $ as $\alpha \to 1^-$ and for $\alpha =1/2$ we get $p=2$. We also assume that $w$ is a radial function $w=w(r)$. We get
\begin{equation}\label{eq.ODE}
rw'(r)=\alpha d w(r)^p-dw(r)=dw(r)\left( \alpha w(r)^{p-1}-1\right).
\end{equation}
There is an equilibrium point $w_*=(1/\alpha )^{1/(p-1)}=\alpha ^{-(1-\alpha )/\alpha }$ (for $\alpha =1/2$ we get $w_*=2$). This gives rise to the constant solution that is also found  in the limit case of linear mobility. But in the case of linear mobility we have $\alpha =1$, $p=1$ and there is not preferred critical value for \eqref{eq.ODE}.

\medskip

 Actually, the existence of the critical value for $0<\alpha <1$ allows us to construct solutions in the following region $D=\{(r,v): r>0, 0<w<w_*\}$ of the ODE phase plane. It is clear that $D$ is an invariant region; it is  bounded by
the solutions $w=0$ and $w=w_*$ from below and above.

\paragraph{Quantitative analysis of \eqref{eq.ODE}.} An asymptotic analysis as $r\to \infty$ gives for all possible solutions
$rw'(r)\sim -dw(r)$ so that $w(r)\sim r^{-d}$ and the original profile $v$ behaves as
$$
F(r)\sim  r^{-dp} \qquad \mbox{as \ } r\to\infty.
$$
Since $dp=d/(1-\alpha )>d$ this tail is integrable. As for the limit $r\to 0$ the only admissible option is to enter the corner point so that
 $$
F(0^+)  = w_*^{p}=\alpha ^{-\alpha }.
$$
Hence, all the solutions in this region will have the same behaviour at $r=0$ to zero order. They are all decreasing and positive for $r>0$.

\paragraph{Explicit expression for $F$.} An explicit computation is possible as follows. Since we have by \eqref{eq.sss1} that
$$
r^{1-d}(r^{d}w)'= \alpha d w^p .
$$
If we  define $z=r^{d}w$, then we  get the ODE $z'= A r^{-a-1}z^p$ with
$$
A=\alpha d, \qquad a=dp-(d-1)-1=d(p-1)=d\alpha /(1-\alpha ).
$$
 Integration of $z^{-p}z'= Ar^{-a-1}$ gives
$$
-\frac1{p-1}(z^{-(p-1)}-C_1)=-\frac{A}{a}r^{-a},
$$
$$
z^{-(p-1)}=C_1+\frac{A(p-1)}{a}r^{-a}.
$$
We have $A(p-1)/a=\alpha $ so that
$$
w(r)=(C+ \alpha  r^{-d(p-1)})^{-1/(p-1)}r^{-d} =
\frac1{\left(\alpha + Cr^{d(p-1)}\right)^{1/(p-1)}},
$$
where $ 1/(p-1)=(1-\alpha )/\alpha  $ ranges in $(0,\infty)$. Finally, the profile $F=w^p$ is given by
\begin{equation}\label{eq:sss profile}
F(r)= F(0^+)\left(1 + Cr^{d\alpha /(1-\alpha )}\right)^{-1/\alpha }, \qquad F(0^+)=\alpha^{-1/\alpha},
\end{equation}
where the exponent $1/\alpha =p/(p-1)$ ranges in $(1,\infty)$ and $C$ is left to be determined.

\medskip

We will later show, by a different method, that, under the additional condition
\begin{equation*}
	\int_{\mathbb R^d} U(t,x) \diff x = \int_{\mathbb R^d} F(|y|) \diff y  = 1,
\end{equation*}
we deduce that the self-similar profile is
\begin{equation*}
	F(|y|) =  \left(  \left( \frac{ \omega_d |y|^d}{ \alpha  }   \right)^{\frac {\alpha} { 1- \alpha}} + \alpha  \right)^{-\frac 1 \alpha}.
\end{equation*}
Hence, the self-similar solution \eqref{eq:solution is self-similar form} of mass $1$ is given by
\begin{equation}
	\label{eq:sss}
	U(t,x) = t^{-\frac 1 \alpha} \left( \alpha +
	 \left( \frac{ \omega_d |x|^dt^{-\frac 1 {\alpha}} } { \alpha} \right)^{\frac {\alpha} {1-\alpha }}
	\right)^{-1/\alpha }.
\end{equation}

\begin{remark}
	\begin{enumerate}
	 \item Self-similar solutions of mass $M$ can be obtained by the rescaling
	\begin{equation*}
		U_M(t,x) = M U(M^\alpha t, x).
	\end{equation*}
	Going back to \eqref{eq:solution is self-similar form}, the profile of the solution of mass $M$ is given by
	\begin{equation*}
		F_M( |y| ) = F\left( \frac {|y|}{M^{\frac 1 d}}  \right)
	\end{equation*}
	which yields solutions of the form
	\eqref{eq:self-sim solution of mass M}
The initial datum of such solution is a Dirac delta. The whole class reminds us of the Barenblatt solutions of fast diffusion equations, cf. \cite{VazBk2006}.
Notice that for large $y$ we have
\begin{equation*}
	F_M (|y|) \sim  \left( \frac{\alpha M }{\omega_d |y|^d} \right) ^{\frac{1}{1-\alpha}} \qquad \text{for } |y| \gg 1
\end{equation*}
so the tail depends on the total mass, unlike in the Fast Diffusion Equation, where the constant for the tail is uniform (see \cite{Vazquez2007}). On the other hand,
$F_M (0) = F(0)$ for all $M$, i.e., near $y = 0$, the self-similar solution does not detect the mass. Notice that, as
\begin{equation*}
	F_M (y) \to F(0), \qquad \text{as } M \to +\infty.
\end{equation*}
In particular the constant value $F(0)$ is the self-similar profile of the global supersolution ${\widetilde u} (t) = (\alpha t)^{-1/\alpha}$, which has infinite mass.

\item  The formula for $\alpha =1/2$ is
$$
F(|y|)= \frac{1}{4\left(1 + C|y|^{d}\right)^{2}}
$$
and the self-similar solution in original variables is
$$
U(t,x)= \frac{t^{-2}}{4\left(1 + C|x|^{d} t^{-2}\right)^{2}}=
\frac{t^{2}}{4\left(t^{2} + C|x|^{d} \right)^{2}},
$$
which is the Cauchy distribution in $d=1$ and the stereographic projection to some sphere in dimension  $d=2$.

\item The selfsimilar solution is a $C^\infty$ solution in space and time. This regularity will not be achieved by the general class of solutions we will describe below, where Lipschitz continuity will be the rule.

\item In the limit $\alpha \to 1$ we obtain the expanding vortex solution described in \cite{Serfaty+Vazquez2014}, given by
$$
U(t,x)=\frac1t \chi_{B(0,R_1 t)} ,\quad \mbox{ with }  \omega_d R_1^{d}= M .
$$
This limit is illustrated in \Cref{fig:sss}.
\end{enumerate}
\end{remark}

\begin{figure}[ht]
		\centering
		\includegraphics[scale=.5]{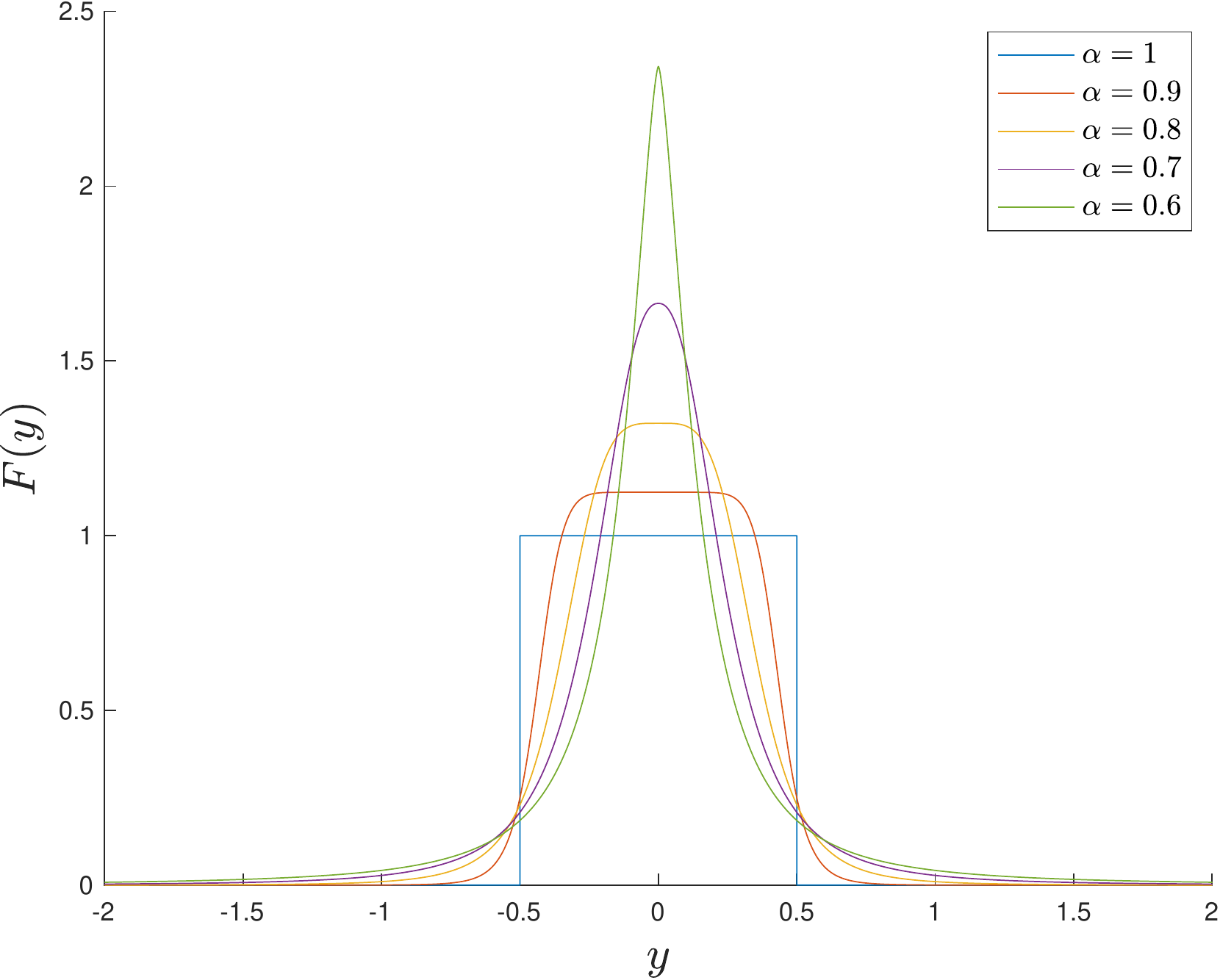}
		\caption{Selfsimilar profiles for  $d=1$ and different values of $\alpha$}
		\label{fig:sss}
	\end{figure}

\section{Mass function of radial solutions}
\label{sec.mass}

In order to proceed with our mathematical analysis, we restrict consideration to radially symmetric solutions and introduce an important tool, the mass function.

\subsection{A PDE for the mass}

\label{sec:PDE mass}

\subsubsection{Radial coordinates}
Let us consider $u = u(t,|x|)$ a radial function and let us define its mass in radial coordinates as
\begin{equation*}
	m (t,r) =d \omega_d \overline m (t, r), \qquad \text{with }\overline m(t,r) = \int_0^r  u(t,\tau) \tau^{d-1} d\tau.
\end{equation*}
We have that
$
	\overline m_r (t,r) = r^{d-1} u(t,r)
$.
Taking derivative in t
\begin{align*}
	\overline m_t &= \int_0^r \tau^{d-1} u_t \diff \tau = \int_0^r \tau^{d-1} \frac{1}{\tau^{d-1}}\frac{\partial }{\partial  r} \left( \tau^{d-1} u^\alpha \frac{\partial v}{\partial r} \right) \diff \tau\\
	&= r^{d-1} (\overline m_r r^{-(d-1)})^\alpha \frac{\partial v}{\partial  r}
	= r^{-\alpha(d-1)} \overline m_r^\alpha r^{d-1} \frac{\diff v}{\diff r}.
\end{align*}
Since $u$ is radial, then $v$ is also radial and its equation can be written
\begin{equation*}
	-\frac{1}{r^{d-1}}\frac{\partial }{\partial r}\left(  r^{d-1} \frac{\partial v}{\partial r}   \right) = u.
\end{equation*}
Hence
\begin{equation*}
	- r^{d-1} \frac{\partial v}{\partial r} = \int_0^r \tau^{d-1} u = \overline m.
\end{equation*}
Therefore, we can write a first order equation for $m$ of the form:
\begin{equation}
	\label{eq:mass Burgers r}
	\overline m_t + r^{-\alpha(d-1)} \overline m \, \overline m_r^\alpha = 0,
\end{equation}
which looks like a difficult variation of the classical Burger's equation.

\subsubsection{Volume coordinates}
Equation \eqref{eq:mass Burgers rho} above includes an unwelcome $r^{-\alpha(d-1)}$. However, by choosing the volume-scaling coordinates
\begin{equation}
	\rho = \omega_d r^{d}.		
\end{equation}
We can write $\overline m_r = d \omega_d r^{d-1} \overline m_\rho $ and hence
\begin{equation*}
	\overline m_t + r^{-\alpha(d-1)}\overline m \, ( d \omega_d r^{d-1} \overline m_\rho )^\alpha = 0.
\end{equation*}
In particular, multiplying by $d \omega_d$ we have
 \begin{equation*}
 		(d \omega_d \overline m)_t +  (d \omega_d \overline m )\, ( d \omega_d \overline m_\rho )^\alpha = 0.
 \end{equation*}
 Changing back to the $m$ variable
\begin{equation}
	\label{eq:mass Burgers rho}
	m_t + m m_\rho^\alpha = 0.
\end{equation}
In this variable, the equation for $m$ does not depend on $d$ anymore. This is a surprising new version of Burgers' equation, which is not in divergence form. For $\alpha\ne 1$, to our knowledge there is no reference in the mathematical literature to this equation.

\subsection{Method of generalised characteristics}
\label{ssect.gch}

The method of generalised characteristics (see \cite[Section 3.2]{Evans1998}) for a generic first order equation
\begin{equation*}
	G(Dw, w, \vv y) = 0,
\end{equation*}
where $\vv y = (t,x)$ consists on constructing parametric characteristic $\vv y(s)$ that can be solved independently. Applying this theory to \eqref{eq:mass Burgers rho}, with the notation $\vv y = (t,r)$, $w = m$, $p_1=m_t$, $p_2=m_\rho$, so that our equation becomes
$G=0$ with
\begin{equation}
	G(p,z,\vv y) = p_1 + z p_2^\alpha,
\end{equation}
we deduce
\begin{theorem}
	\label{thm:characteristics}
	Let $m$ be a classical solution of \eqref{eq:mass Burgers rho} with initial data $m_0$, and let the derivative be called $u_0 = (m_0)_\rho \ge 0$. As long as the characteristics
	\begin{align}
		\label{eq:characteristic}
		\rho(t) &= \rho_0 + \alpha m_0(\rho_0) u_0(\rho_0)^{\alpha - 1} t
	\end{align}
	do not cross,
	the solution is given by
	\begin{equation}
		\label{eq:mass characteristics}
		m(t, \rho(t)) = m_0(\rho_0) (1 + \alpha u_0( \rho_0 )^\alpha t)^{1- \frac 1 \alpha }
	\end{equation}
	and its derivative $u = m_\rho$ by
	\begin{equation}
		u(t,\rho(t)) = (u_0(\rho_0)^{-\alpha} + \alpha t)^{-\frac 1 \alpha}.
	\end{equation}
\end{theorem}
\begin{remark}
	\label{rem:comments on characteristics}
\begin{enumerate}
	\item Notice that characteristics are always straight lines. Recall that $\rho$ is a volume variable.
	
	\item Due to our choice of coordinates $ \omega_d r^{d-1} u = m_r = d \omega_d r^{d-1} m_\rho$ so $m(\rho) = \int_0^\rho u(s) \diff s$.
	
	\item The equation for the mass \eqref{eq:mass Burgers rho} has infinite speed of propagation for the derivative $u$.
	
	\noindent \label{it:u0 triangle} When $u_0$ is the triangle
	$
		u_0(\rho) = (1-\rho)_+
	$
	then the mass is given by
	\begin{equation*}
		m_0(\rho) = \frac 1 2 (1 - (1 - \rho_0)^2 ) \qquad \text{for } 0 \le  \rho_0 \le  1.
	\end{equation*}
	Hence the characteristics from $\rho_0 \in [0,1]$ are written
	\begin{equation*}
		\rho = \rho_0 + \frac \alpha 2  \left( 1 - (1 - \rho_0)^2 \right) (1-\rho_0)^{\alpha-1} t \qquad \text{for } 0 \le  \rho_0 \le  1.
	\end{equation*}
	For any $ t > 0$, these characteristics cover all $\rho \ge 0$, as shown in \Cref{fig:characteristics triangle}.
	\begin{figure}[H]
		\centering
		\includegraphics[]{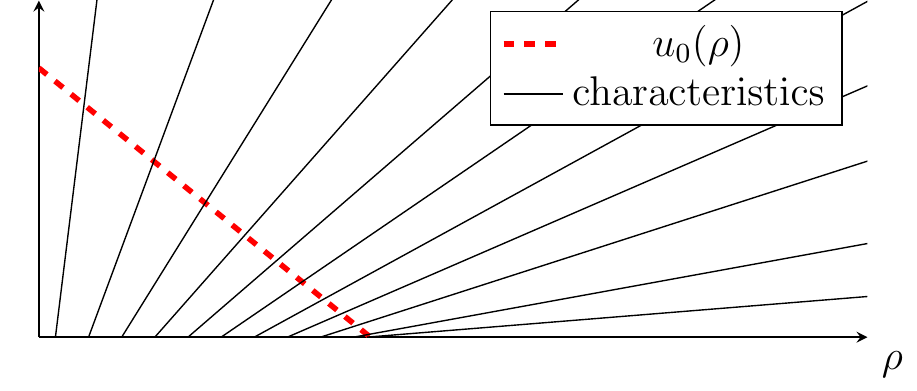}
		\caption{Characteristics corresponding to $u_0(\rho) = (1-\rho)_+$ for $\rho_0 \in [0,1]$}
		\label{fig:characteristics triangle}
	\end{figure}
	Notice that characteristics from $ \rho_0 \in (1,+\infty)$ are constant $\rho = \rho_0$, since $u_0(\rho_0)$ and $m(\rho_0) = 1/2$.

	\item \label{it:u0 two triangle} A remarkable difference of \eqref{eq:mass Burgers rho} with respect to Burger's equation is the fact that, even for Lipschitz initial data $m_0$, characteristics may cross for all $t > 0$ (see \Cref{fig:characteristics double triangle}).
	These intersections will lead to a shock, governed by a variant of the classical Rankine-Hugoniot conditions \cite{GR,Sm}, as we will see below.
	\begin{figure}[H]
		\centering
		\includegraphics[]{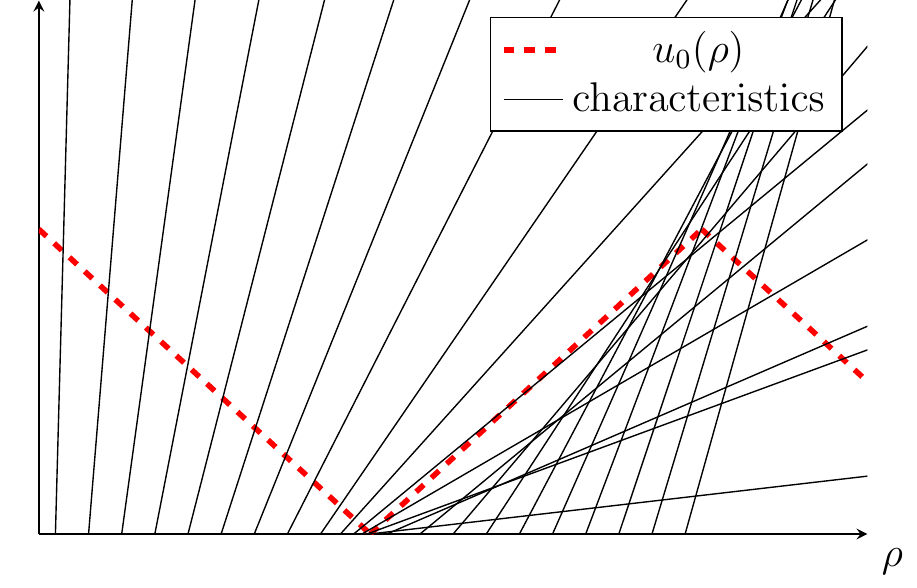}
		\caption{Characteristics corresponding to $u_0(\rho) = (1-\rho)_+ + (  1 - |\rho-2|  )_+$ for $\rho_0 \in [0,2]$.}
		\label{fig:characteristics double triangle}
	\end{figure}

	\item Notice that by a happy coincidence, $u = m_\rho = p_2$. Hence, by solving the system of ODEs, we already obtain the value of the original function $u$ along the characteristic.
	
	\end{enumerate}
\end{remark}

\begin{proof}
As usual, we form a two-parametric family $\vv y(s,\rho)$ of characteristics and then identifying the surface they build as the solution.
Following \cite[Section 3.2]{Evans1998} we next construct the characteristics. Using the notations
\begin{equation*}
	z(s) = w(\vv y(s)), \qquad \vv p(s) = Dw (\vv y(s)).
\end{equation*}
The equations for the characteristics are
\begin{subequations}
	\label{eq:characteristics Evans}
	\begin{align}
		\dot {\vv p}(s) &= - D_y G (\vv p(s), z(s), \vv y(s)) - D_z G (\vv p(s), z(s) , \vv y(s)) \vv p (s) \\
		\dot z(s) &= D_p G (\vv p(s), z(s) , \vv y(s))  \cdot \vv p (s) \\
		\dot{ \vv y }(s)  &= D_p G (\vv p(s), z(s) , \vv y(s))
	\end{align}
\end{subequations}
We write system \eqref{eq:characteristics Evans}
as
\begin{subequations}
	\label{eq:characteristics Evans2}
	\begin{align}
		\dot {p_1}(s) &= - p_2^{\alpha} p_1\\
		\dot p_2 (s) &=  - p_2^{\alpha + 1} \\
		\dot z(s) &= (1, \alpha z p_2^{\alpha - 1})\cdot (p_1, p_2)  = p_1 + \alpha z p_2^\alpha \\
		\dot t (s) &= 1\\
		\dot \rho (s) &= \alpha z p_2^{\alpha - 1}
	\end{align}
\end{subequations}
For the initial data we take $t(0) = 0 $ and
$\rho(0)=\rho_0>0$. We have the following initial conditions:
\begin{subequations}
	\label{eq:initial conditions z and p2}
	\begin{align}
	z(0) &= z_0 = m (0,\rho_0)  \int_0^{\rho_0} u_0 (s) ds. \\
	p_2 (0) &= p_{2,0} = m_r (0,\rho_0) = u (0,\rho_0) = u_0 (\rho_0).
\end{align}
\end{subequations}
The equation relates the values of  $p_1(0), p_{2,0}$ and $z_0$:
$
	p_1(0) = - p_{2,0}^\alpha z_0.
$

We first notice that $t(s) = s$. If $u_0 (\rho_0) = 0$ then $p_1 = p_2 = 0$, and hence $z \equiv z_0$ and $\rho(s) \equiv \rho_0$. In other words, points outside the support of $u_0$ do not propagate in any direction. Furthermore, if $u_0(r) = 0$ then $m(r,t) = m(r,0)$. However, if $u_0 (\rho_0) = p_{2,0} > 0$ then $p_2(t) > 0$ then $\dot \rho > 0$ and hence it is increasing.

Observe that the equation for $p_2$ is autonomous, then it can be solved explicitly to get
\begin{equation}
	p_2(t) =  \left(p_{2,0}^{-\alpha }+ \alpha t\right)^{-1/\alpha },
\end{equation}
if $p_{2,0} > 0$.
Notice that $p_1 \dot p_2 - \dot p_1 p_2 = 0$,
therefore $p_1/ p_2$ is constant and
\begin{equation*}
	p_1 (t) = \frac{p_1(0)}{p_{2,0}} \left(p_{2,0}^{-\alpha }+ \alpha t\right)^{-1/\alpha }.
\end{equation*}
Using the condition on the initial data we finally obtain
\begin{equation}
	p_1 (t) = -\frac{z_0}{p_{2,0}} \left(1+\alpha p_{2,0}  t\right)^{-1/\alpha
   }.
 \end{equation}
Once $p_1$ and $p_2$ are known, we can solve for $z$ as a linear equation with variable coefficients. We have
$$
\dot z(t) -\frac{\alpha}{p_{2,0}^{-\alpha }+ \alpha t}z(t)=p_1(t).
$$
Since, for any two functions $f' g - f g' = (f/g)' g^2$, we deduce that
$$
\frac{d}{dt} \left( \frac{z(t)}{(p_{2,0}^{-\alpha }+ \alpha t)^{-1}}\right)=-\frac{z_0}{p_{2,0}} \left(1+\alpha p_{2,0}^\alpha  t\right)^{-1-1/\alpha}.
$$
Integrating on $[0,t]$ and solving for $z$ we deduce that
\begin{align}
		z(t) &= {z_0 }\left(1+ \alpha p_{2,0}^\alpha  t\right)^{1-1/\alpha }.
	\end{align}
Thus we deduce
\begin{equation*}
	\dot \rho(t) = z(t) p_2 (t)^{\alpha-1} =  z_0 p_{2,0}^{\alpha - 1}.
\end{equation*}
Hence
\begin{align}
	\label{eq:characteristic r explicit}
	\rho(t)
	&=\rho_0 + \alpha z_0 p_{2,0}^{\alpha - 1} t.
\end{align}
To deduce \eqref{eq:characteristic} we substitute the values from the initial data in \eqref{eq:initial conditions z and p2}.
\end{proof}

\begin{remark}
Notice that the argument works for any $\alpha > 0$.
The characteristics' formula \eqref{eq:characteristic} shows us that the cases $\alpha \in (0,1), \{1\}, (1,+\infty)$ behave quite differently. The case $\alpha = 1$ is the Burgers equation.
In the case $0 < \alpha < 1$ solutions with small positive initial value will disperse almost instantaneously (as in the Fast Diffusion Equation, see \cite{Vazquez2007}). Oppositely, for $\alpha > 1$ the larger the initial data the slower it will diffuse (as in the Porous Medium Equation, see \cite{Vazquez2007}).
\end{remark}

\begin{remark}
	Notice that for points in the support of $u_0$, characteristics are increasing straight lines. If the support of $u_0$ is bounded, characteristics coming from the support of $u_0$ (with positive values of $u$), will intersect characteristics from outside the support. We will see later how solutions overcome this difficulty.
\end{remark}

\section{Radial non-increasing data \texorpdfstring{$u_0$}{u0}}
\label{sec.radial}

In this section we will consider with non-increasing radial data for \eqref{eq:main equation}, by the method of characteristics.

\subsection{Continuous \texorpdfstring{$u_0$}{u0}}
\label{ssec:characteristics decreasing continuous}

In this case we will show that the characteristics do not cross, and hence we can construct classical solutions of \eqref{eq:mass Burgers rho} using Theorem \ref{thm:characteristics}. Then $u$ is determined in an implicit way by
\begin{align}
	\label{eq:solution radial nonincreasing continuous data}
	u\left(t, x \right) &= \left(u_0 \left(\rho_0\right)^{-\alpha} + \alpha t \right)^{-\frac {1} \alpha}  \quad \textrm{ where } \omega_d {|x|^d} = \rho_0 + \alpha m_0\left( \rho_0 \right) u_0\left( \rho_0 \right)^{\alpha - 1} t .
\end{align}
We introduce the following function
\begin{equation*}
	P_t(\rho_0) = \rho_0 + \alpha m_0\left( \rho_0 \right) u_0 \left( \rho_0 \right)^{\alpha - 1} t.
\end{equation*}
Let us distinguish two cases.

\paragraph{Positive $u_0$.}
Since $u$ is positive, $m$ is strictly increasing and, since $u_0$ is strictly decreasing, then $u_0 (\rho_0)^{\alpha-1}$ is non-decreasing. Hence, for every $t > 0$, $P_t$ is a strictly increasing function of $\rho_0$, and therefore invertible. It is null at zero and unbounded at infinity. Hence, $P_t : [0,+\infty) \to [0,+\infty)$ is invertible. Therefore, for every $t > 0$ and $x \in \mathbb R^d$ there exists a unique $\rho_0$ such $\omega_d |x|^d = P_t (\rho_0)$.

\paragraph{Compactly supported $u_0$.}

If the initial datum reaches zero, then $\supp u_0 = \overline{B_R}$ for some $R > 0$. Then $P_t$ is still a strictly increasing function for $\rho < R$. Clearly $P_t(R^-) = +\infty$. Hence, for every $t > 0$ we have $P_t : [0,R) \to [0,+\infty)$ is invertible.

\subsection{Discontinuous data: rarefaction fan solutions}
\label{ssec:characteristics decreasing discontinuous}

\subsubsection{Rarefaction fan solution for $u_0 (\rho) = \chi_{[0,L]} (\rho)$}
	\label{sec:rarefaction fan u_0 square}
	If one considers a regularised version of the square functions
	\begin{equation}
		u_0^{(\ee)} (\rho)= \begin{dcases}
			c_0 & \rho \le L  , \\
			\frac{c_0}{\ee} (L + \ee  - \rho) & L   \le \rho < L + \ee \\
			0 & \rho \generic L + \ee.
		\end{dcases}
	\end{equation}
	The initial mass becomes
	\begin{equation*}
		m_0^{(\ee)} (\rho) = \begin{dcases}
			c_0 \rho & \rho \le L, \\
			c_0 L + \frac{c_0}{\ee} \frac{(L + \ee - \rho)^2} 2& L < \rho \le L + \ee \\
			c_0 L + c_0 \ee & \rho \ge L + \ee.
		\end{dcases}
	\end{equation*}
	We write the characteristics
	\begin{equation*}
		\rho = \rho_0 + \alpha m_0^{(\ee)} (\rho_0) u_0^{(\ee)}(\rho_0)^{\alpha - 1}  t , \qquad \text{for } 0 < \rho < L+\ee.
	\end{equation*}
	Since $0<\alpha<1$, these characteristics cover the whole space $(t,\rho) \in (0,+\infty)^2$. There is a function $\rho_0^{(\ee)}(t,\rho)$, with no simple explicit formula.	
	Then
	\begin{equation}
		u^{(\ee)}(t,\rho) = ( u_0^{(\ee)}(\rho_0^{(\ee)}(t,r))^{-\alpha} + \alpha t)^{-\frac 1 \alpha}.
	\end{equation}
	The characteristic emanating from the end of the flat part is still
	\begin{equation*}
		\rho = L + \alpha c_0L (c_0)^{\alpha - 1} t = L ( 1 + \alpha c_0^\alpha t).
	\end{equation*}
	For $t>0$ and $\rho > L ( 1 + \alpha c_0^\alpha t)$, as $\ee \to 0$ then
	$
		\rho_0^{(\ee)}(\rho) \to L,
	$
	whereas $u_0^{(\ee)}(\rho_0^{(\ee)}(t,\rho))$ is a bounded sequence, so let $p_{2,0}(t,\rho) \in [0, c_0]$ be a pointwise limit. Then $p_{2,0}(t,\rho)$ is a solution of
	\begin{equation*}
		 \rho = L + \alpha m_0 (L) p_{2,0}(t,\rho)^{\alpha - 1} t.
	\end{equation*}	
	Since $m_0 (L) = c_0 L$ we have that
	\begin{equation*}
		\rho = L + \alpha c_0 L p_{2,0}(t,\rho)^{\alpha - 1} t.
	\end{equation*}
	Therefore the limit is unique and
	\begin{equation*}
		u_0^{(\ee)}(\rho_0^{(\ee)}(t,\rho)) \longrightarrow p_{2,0}(t,\rho) = \left( \frac{\rho-L}{\alpha c_0 L t }  \right)^{\frac 1 {\alpha - 1}}.
	\end{equation*}
	In particular, the pointwise limit is unique. Hence, the whole sequence converges pointwisely to this limit
	\begin{equation*}
		u^{(\ee)}(t,\rho)  \longrightarrow \left(  \left( \frac{\rho-L}{\alpha c_0 L t}   \right)^{\frac \alpha { 1- \alpha}} + \alpha t \right)^{-\frac 1 \alpha}, \qquad \text{for } t>0 \textrm{ and }\rho > L ( 1 + \alpha c_0^\alpha t).
	\end{equation*}
	Since the pointwise limit elsewhere is given by $u^{(\ee)}(t,\rho) \longrightarrow u(t,\rho)$, where
	\begin{equation}
		\label{eq:u square data rarefaction fan}
		 u(t,\rho) = \begin{dcases}
		 \left(   c_0 ^{- \alpha } + \alpha t \right)^{-\frac 1 \alpha} & \text{if } \rho \le L ( 1 + \alpha c_0^\alpha t) \\
			\left(  \left( \frac{ \rho -L}{\alpha c_0 L t}   \right)^{\frac \alpha { 1- \alpha}} + \alpha t \right)^{-\frac 1 \alpha} &  \text{if } \rho > L ( 1 + \alpha c_0^\alpha t).
	\end{dcases}
	\end{equation}
	We have only proved pointwise convergence. Since the sequence $u^{(\ee)} \in L^1 \cap L^\infty$ is bounded, the limit is $L^p$ for $1 < p < +\infty$ and weak-$\star$ in $L^\infty$.

	\begin{remark}
	This last function is continuous, and therefore its mass is a classical of \eqref{eq:mass Burgers rho} .
	\end{remark}
	\begin{figure}[H]
	\centering
		\includegraphics[width=.49\textwidth]{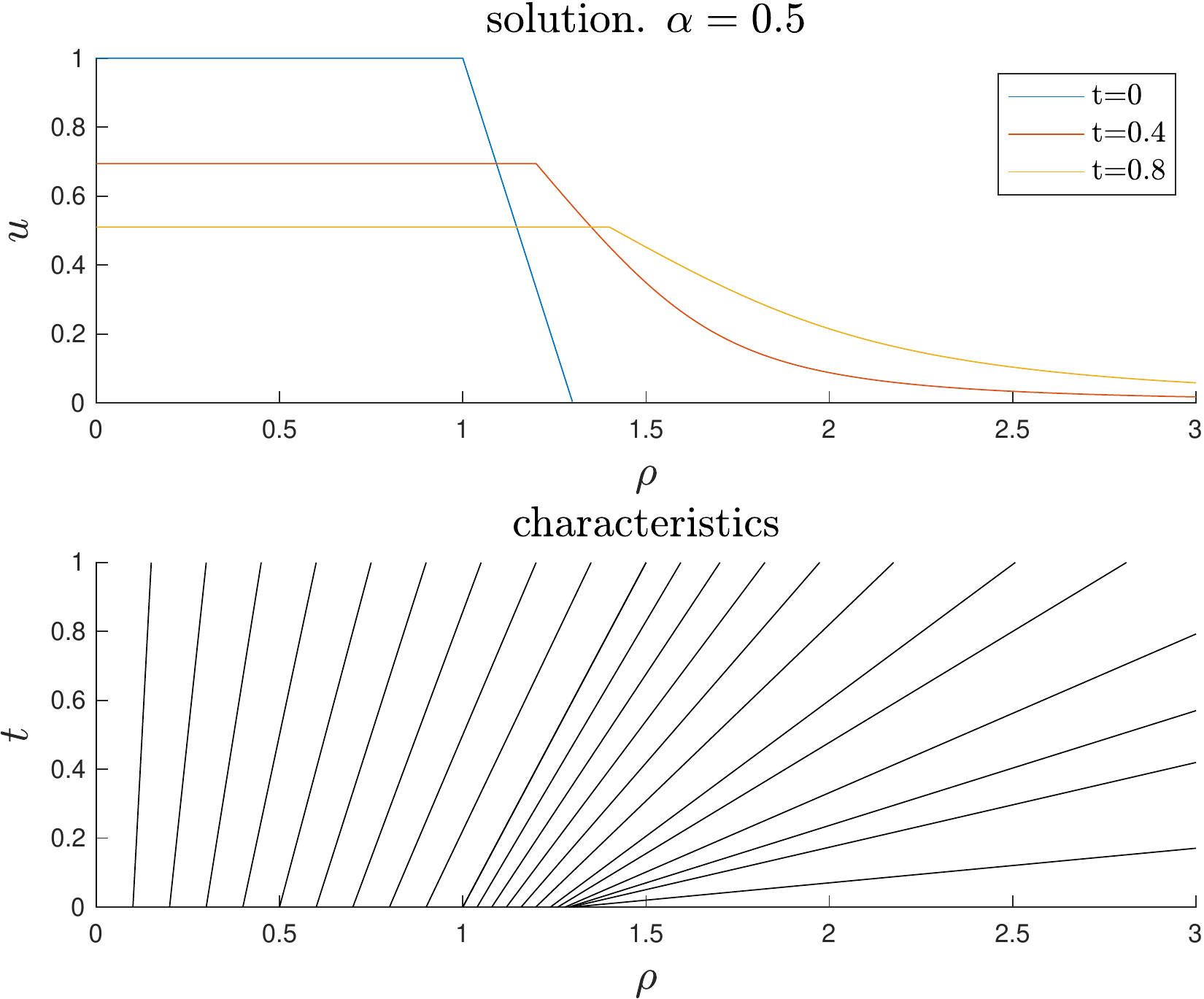}
		\includegraphics[width=.49\textwidth]{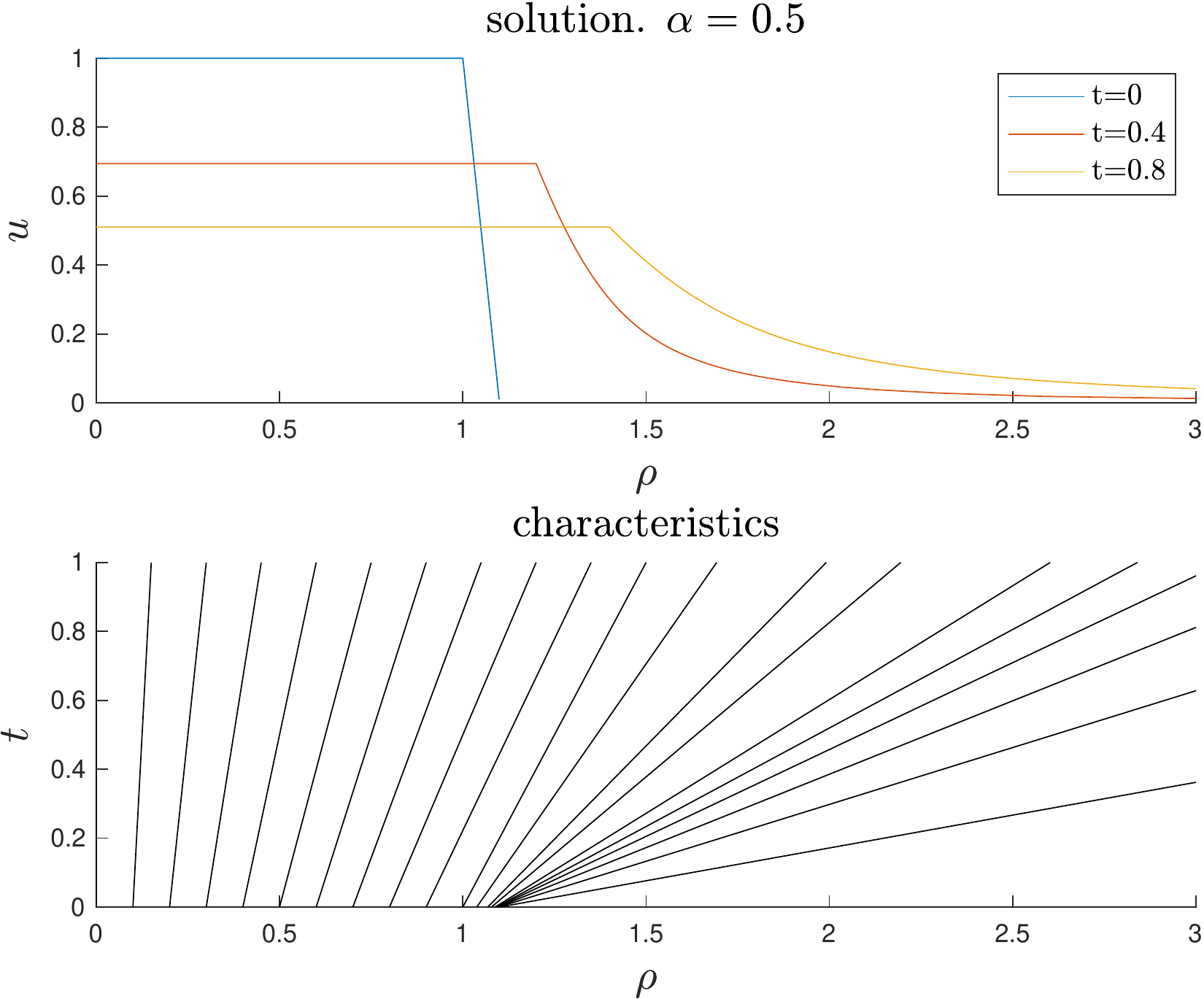}
		\caption{Solution by characteristics for the case $\alpha = .5$ and initial data  $u_0^{(\ee)}$ ($c_0 = 1$, $L =1$ and $\ee=0.3$ left and $\ee = 0.1$ right) a continuous version of the characteristic function, still with compact support. The characteristics guarantee that flat zones are preserved. For $t>0$ there is no longer compact support.}
		\label{fig:alpha-0.5-quasisquare}
	\end{figure}
	
	We point out that an analogy to fast diffusion equation is limited since Figure \ref{fig:alpha-0.5-quasisquare} shows that solutions are Lipschitz continuous and no more even if fat tails are produced.

\subsubsection{Recovering the self-similar solution}
	\label{sec:from square to self-similar}
	
	Let us consider \eqref{eq:u square data rarefaction fan} and fix the total mass $M = c_0 L$.  We get
	\begin{equation*}
		 u_L(t,x) = \begin{dcases}
		 \left(   M ^{- \alpha } L^\alpha + \alpha t \right)^{-\frac 1 \alpha} & \omega_d{|x|^d} \le L + \alpha M^\alpha L^{1 - \alpha} t \\
			\left(  \left( \frac{\omega_d |x|^d-L}{\alpha M t}   \right)^{\frac \alpha { 1- \alpha}} + \alpha t \right)^{-\frac 1 \alpha} & \omega_d {|x|^d} > L + \alpha M^\alpha L^{1 - \alpha} t.
	\end{dcases}
	\end{equation*}
	As $L \to 0$ we recover the self-similar solution of mass $M$
	\begin{align*}
		u_L(t,x) &\longrightarrow \left(  \left( \frac{\omega_d |x|^d}{ \alpha M t}   \right)^{\frac \alpha { 1- \alpha}} + \alpha t \right)^{-\frac 1 \alpha}
		= t^{-\frac 1 \alpha}\left(  \left( \frac{t^{-\frac 1 \alpha} \omega_d |x|^d}{\alpha M}   \right)^{\frac \alpha { 1- \alpha}} + \alpha  \right)^{-\frac 1 \alpha}\\
		&= t^{-\frac 1 \alpha} F \left(  \frac{t^{-\frac 1 {d\alpha}}|x|}{M^{\frac 1 d}}  \right) = U_M (t,x).
	\end{align*}
	
	\subsubsection{Rarefaction fan solution for discontinuous non-increasing data}
\label{ssect.rf}
	
	Combining the formula of solutions by characteristics \eqref{eq:solution radial nonincreasing continuous data} with the rarefaction fan idea we construct
	\begin{subequations}
		\label{eq:solution radial nonincreasing discontinuous data}
	\begin{equation}
	u\left(t, \rho  \right) = \left(\eta_0^{-\alpha} + \alpha t \right)^{-\frac {1} \alpha}  \quad \textrm{where } \rho = \rho_0 + \alpha m_0\left( \rho_0 \right) \eta_0^{\alpha - 1} t,
	\end{equation}
	where $\eta_0$ is some value
	\begin{equation}
		\eta_0 \in \left[{u_0} \left(\rho_0^+ \right), {u_0} \left(\rho_0^- \right) \right] .
	\end{equation}
	\end{subequations}
	Let us now show that these solutions are well-defined.
	
	\begin{proposition}
		\label{eq:rarefaction fan solution discontinuous}
		Let $u_0 \in L^\infty_+ ([0,+\infty)) \cap L^\infty_+ ([0,+\infty)) $, $u_0 \not \equiv 0$ and radially non-increasing.
		For every $t > 0$, the map
		\begin{eqnarray}
			P_t :  \bigcup_{\rho_0 : u_0 (\rho_0) > 0 } \{\rho_0\} \times \left[{u_0} \left(\rho_0^+ \right), {u_0} \left(\rho_0^- \right) \right]  & \longrightarrow & \mathbb R_+ \\
			(\rho_0, \eta_0, t) & \longmapsto & \rho_0 + \alpha m_0 (\rho_0) \eta_0^{\alpha - 1} t
		\end{eqnarray}
		is bijective. Therefore, the map \eqref{eq:solution radial nonincreasing discontinuous data} is well-defined. Furthermore, it defines a function $u \in \mathcal C([0,+\infty)^2)$ such that
		\begin{equation}
			u(t, \cdot) \to u_0 \textrm{ in } L^1 (\mathbb R^n).
		\end{equation}
		The function $m$ given by
		\begin{equation}
		\label{eq:mass rarefaction fan decreasing data}
			m(t,\rho) = m_0 (\rho_0) (1 + \alpha \eta_0^\alpha t)^{1- \frac 1 \alpha} , \qquad \text{where } (\rho_0, \eta_0) = P_t^{-1} (\rho).
		\end{equation}
		is a classical solutions of \eqref{eq:mass Burgers rho} by characteristics.
	\end{proposition}

	\begin{proof}
		We define an order over the domain of $P_t$
		\begin{equation*}
			(\rho_{0,1}, \eta_{0,1}) \prec (\rho_{0,2}, \eta_{0,2}) \equiv
			\begin{dcases}
				\rho_{0,1} < \rho_{0,2}  \\[2mm]
				\text{or} \\
				\rho_{0,1} = \rho_{0,2} \text{ and } \eta_{0,1} > \eta_{0,2}.
			\end{dcases}
		\end{equation*}
		This defines a strict total order in the domain of $P_t$. Furthermore, notice that
		\begin{equation*}
			(\rho_{0,1}, \eta_{0,1}) \prec (\rho_{0,2}, \eta_{0,2}) \implies P_t(\rho_{0,1}, \eta_{0,1}) < P_t(\rho_{0,2}, \eta_{0,2}).
		\end{equation*}
		Hence, $P_t$ is injective. Furthermore, it is continuous with the topology induced in the domain of $P_t$. It is immediate to check that
		\begin{equation*}
			P_t (0,u(0^+)) = 0.
		\end{equation*}
		Notice that $\{ \rho_0 : u_0(\rho_0) > 0\} = (0,R)$ where $R \le +\infty$. As $\rho_0 \nearrow R$ we have that $u_0(\rho_0^+) \to 0$ and $m_0(\rho_0) \nearrow M $, hence
		\begin{equation*}
			P_t (\rho_0,u(\rho_0^+)) \nearrow +\infty.
		\end{equation*}
		Hence $P_t$ is surjective. This completes the proof.
	\end{proof}

\subsubsection{Data with an initial gap}
	\label{sec:data with initial gap}
		Notice that if $u$ is given by \eqref{eq:solution radial nonincreasing discontinuous data}, then
		\begin{equation*}
		\widetilde u (t,\rho) =
		\begin{dcases} u(t,\rho-b) & \rho \ge b, \\
		0 & 0 \le \rho < b	
		\end{dcases}
		\end{equation*}
		is also a solution, and it corresponds to the initial datum
		\begin{equation*}
			\widetilde u_0 (\rho) =
			\begin{dcases} u_0(\rho-b) & \rho \ge b, \\
			0 & 0 \le \rho < b	
			\end{dcases}
		\end{equation*}
		Furthermore, this solution can be obtained by approximation by continuous initial data given by characteristics. Therefore $\supp \widehat u(t,\cdot) = [b,+\infty)$. The conclusion is that this kind of gap is preserved in time.
		\begin{figure}[H]
			\centering
			\includegraphics[width=.49\textwidth]{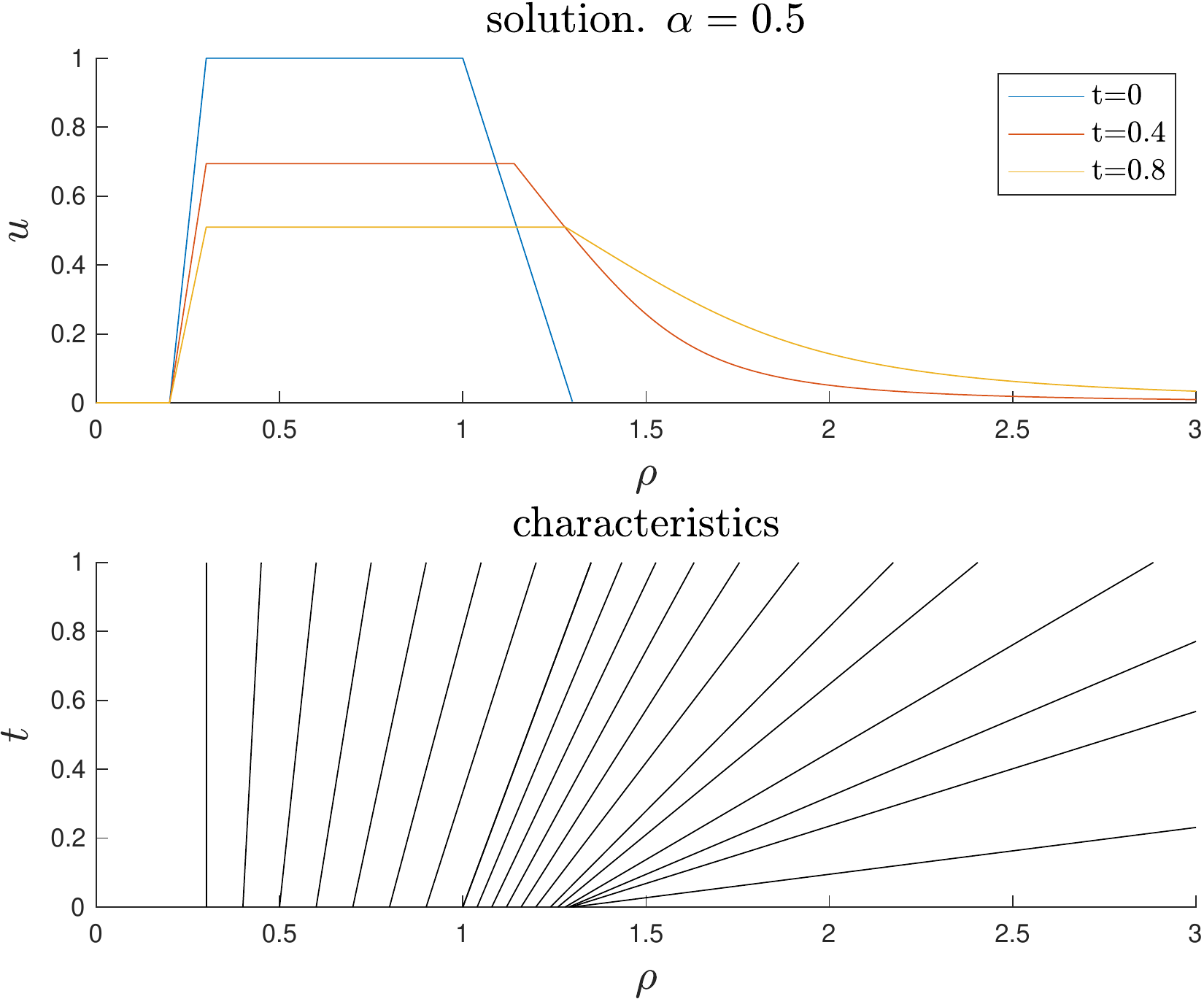}
			\includegraphics[width=.49\textwidth]{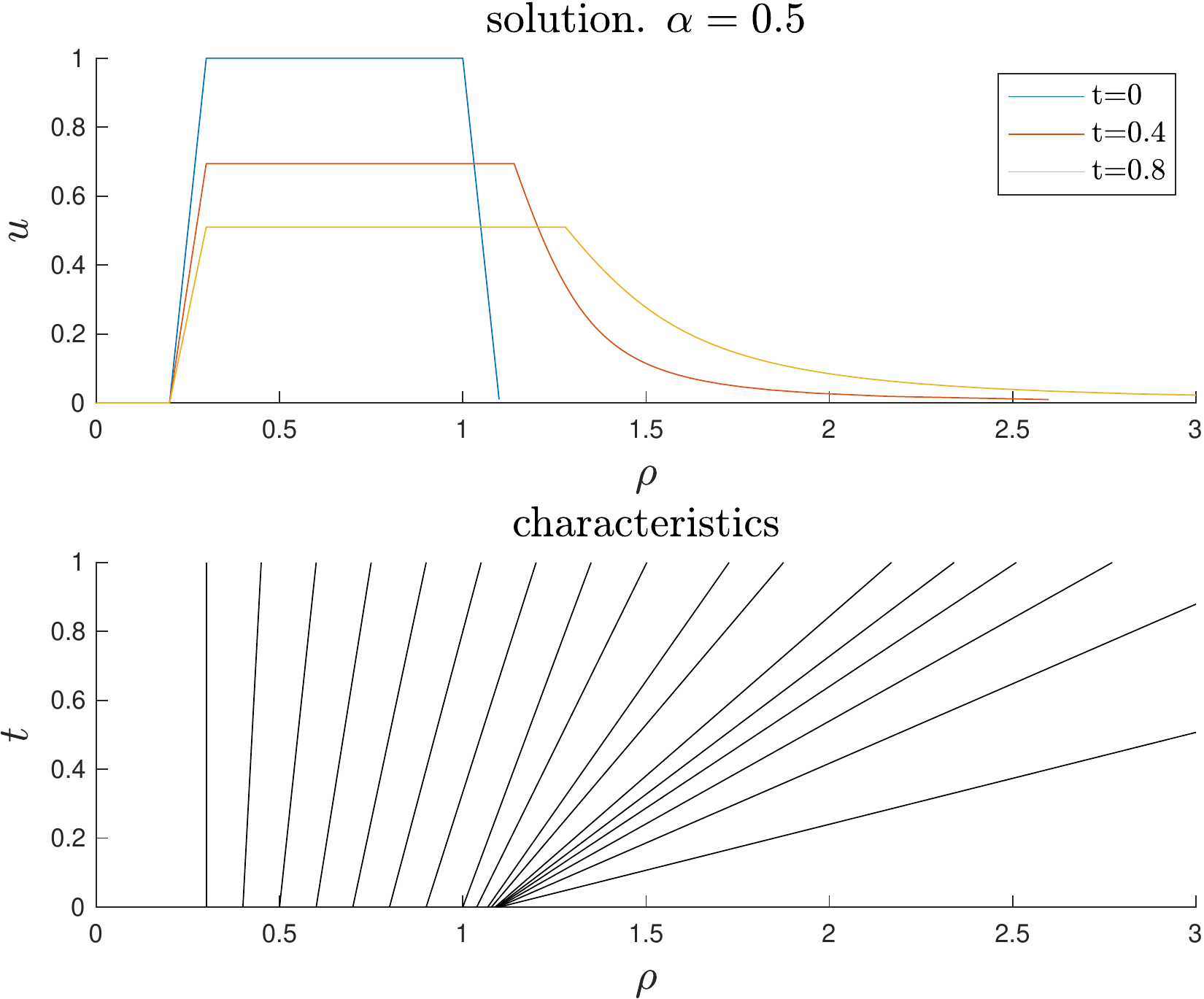}
			\caption{Solution by characteristics for the case $\alpha = .5$ where initial data has an initial gap.}
			\label{fig:alpha-0.5-quasisquare-gap}
		\end{figure}

	\begin{remark}
		\label{rem:P_t as rho to infinity}
		Notice that, if $\supp u_0 = [b,L]$, then for any $t > 0$, $\lim_{\rho \to +\infty} P_t^{-1} (\rho) = (L,0)$. If $\supp u_0 = [0,+\infty)$ then $\lim_{\rho \to +\infty} P_t^{-1} (\rho) \to +\infty$.
	\end{remark}
	
\subsection{Qualitative properties}
\label{ssec:characteristics decreasing qualitative}

\subsubsection{Mass conservation}

\begin{proposition}
	\label{prop:characteristics mass conservation}
	For the classical solution $m$ of \eqref{eq:mass Burgers rho} given by \eqref{eq:solution radial  nonincreasing discontinuous data} under the hypothesis of \Cref{eq:rarefaction fan solution discontinuous}. Then, for every $t \ge 0$, we have that
	\begin{equation}
		\lim_{\rho \to + \infty} m(t,\rho) = \lim_{\rho \to + \infty} m_0(\rho).
	\end{equation}
	\begin{proof}
		Combining \eqref{eq:mass rarefaction fan decreasing data} and \Cref{rem:P_t as rho to infinity} we conclude the result.
	\end{proof}
\end{proposition}

\subsubsection{Comparison principle for $u$ by characteristics}
 Since we have classical solutions by characteristics, we prove a comparison principle for $u = m_\rho$ by direct computation. This immediately implies there exists some kind of comparison principle for $m$.
\begin{theorem}
	\label{thm:comparison principle characteristics}
	Assume that $u_{0,1}, u_{0,2} \in L^\infty (\mathbb R^d)$ radially non-increasing and such that such that $u_{0,1} \ge u_{0,2}$ in $\mathbb R^d$. Let $u_1, u_2$ be given by \eqref{eq:solution radial nonincreasing discontinuous data}. Then $u_1 \ge  u_2$ in $[0,+\infty)\times \mathbb R^d$.
\end{theorem}
\begin{proof}
	Assume, towards a contradiction, that $u_1(t,\rho) < u_2(t,\rho)$ for some $t > 0 , \rho$.
	Since the functions are given by \eqref{eq:solution radial nonincreasing discontinuous data} we write this inequality as
	\begin{equation}
		(\eta_{0,1}^{-\alpha} + \alpha t)^{-\frac 1 \alpha} <(\eta_{0,2}^{-\alpha} + \alpha t)^{-\frac 1 \alpha},
	\end{equation}
	where $\eta_{0,i} \in [u_{0,i} (\rho_{0,i}^+),u_{0,i}(\rho_{0,i}^-)]$.
	Therefore, we have that $\eta_{0,1} < \eta_{0,2}$.
	Since $u_{0,1} \ge  u_{0,2}$ we have that $\rho_{0,2} \le \rho_{0,1}$. Since $m_{0,2} \le m_{0,1}$ we have that $m_{0,2}(\rho_{0,2} ) \le m_{0,1}(\rho_{0,1} )$.
	
	If $m_{0,2}(\rho_{0,2} ) =0$, then $\rho = \rho_{0,2} = 0$, but at $\rho = 0$ we have $u_2(t,\rho) = u_2(t,0) = u_{0,2} (0) \ge u_{0,1} (0) = u_1(t,0)$ and we reach a contradiction.
	
	If $m_{0,2}(\rho_{0,2} ) > 0$, we have that
	\begin{equation}
		\rho = \rho_{0,1} + \alpha m_{0,1} (\rho_{0,1}) \eta_{0,1} ^{\alpha - 1} t  < \rho_{0,2} + \alpha m_{0,2} (\rho_{0,2}) \eta_{0,2} ^{\alpha - 1} t = \rho.
	\end{equation}
	This is a contradiction.
	\end{proof}

\subsection{Asymptotic behaviour}
\label{ssec:characteristics decreasing asymptotic}

For the study of the asymptotic behaviour we consider a rescaled version of the solution $u$ by considering the scaling the self-similar solution
\begin{equation}
	\label{eq:profile of u}
	w(t,y) = t^{\frac 1 \alpha} u( t , t ^{\frac 1 {d \alpha}} y ).
\end{equation}
This is the natural candidate to converge to a stationary non-trivial profile. We will prove stabilisation of the rescaled flow in the strong form of uniform convergence in relative error to the self-similar profile of the solution of same total mass $M$, i.e., $F_M$.

\subsubsection{Asymptotic behaviour as $t \to \infty$ for non-increasing data $u_0$}
\label{sec:asymp in t non-increasing u_0}
We  tackle the general case for solutions given by characteristics and state the convergence in relative error
\begin{theorem}
	\label{thm:asymptotics t non-increasing u_0}
	Let $u_0 \in L^\infty_c (\mathbb R^d)$ radially non-increasing, $M = \| u_0 \|_{L^1}$, and let $u$ be given by \eqref{eq:solution radial nonincreasing discontinuous data}. Then, we have that
	\begin{equation}
		\label{eq:asymp uniform convergence in relative error}
		\sup_{y \in \mathbb R^d} \left| \frac{  w(t,y)   -  F_M \left (  {|y|}  \right ) }{F_M \left ( {|y|}  \right )}\right|  \longrightarrow 0 ,  \qquad \textrm{as } t \to +\infty.
	\end{equation}	
	where $w$ is given by \eqref{eq:profile of u}.
\end{theorem}

As as direct consequence of this theorem, we have the $L^\infty$ convergence with sharp rate
\begin{corollary} Under the same assumptions we have
\begin{equation}
\lim_{t\to\infty}t^{1/\alpha}|u(t,x)-U_M(t,x)|=0
\end{equation}
uniformly in $x\in \mathbb R^d$.
\end{corollary}

We split the proof of Theorem \ref{thm:asymptotics t non-increasing u_0} into several lemmas

\begin{lemma}
	\label{lem:relative error with self-similar power alpha}
	Let $u_0$ be radially nonincreasing and let $ \omega_d |x|^d= \rho = \rho_0 + \alpha m_0 (\rho_0) \eta_0^{\alpha - 1} t $. Then,
	\begin{equation}\label{formula}
		\frac{u(t,x) ^{\alpha} - U_M( t, x  ) ^{\alpha} }{U_M(t,x) ^{\alpha}} =  \left[1 - \left( \frac{ \rho / ( \rho - \rho_0)  }  {  M / m_0(\rho_0) } \right)^{\frac {\alpha} {1-\alpha }}\right] \left[1+\alpha \left( \frac{ \rho-\rho_0  } { \alpha m_0(\rho_0) t ^{\frac 1 \alpha} } \right)^{-\frac {\alpha} {1-\alpha }}  \right]^{-1}.
	\end{equation}
\end{lemma}

\begin{proof}
We have that
	\begin{equation}
		 u(x  , t) = \left(  \eta_0 ^{-\alpha} + \alpha t  \right)^{-\frac 1 \alpha}
	\end{equation}
	where $\eta_0$ is given by \eqref{eq:characteristic}. Going back to \eqref{eq:self-sim solution of mass M} we have
	\begin{equation}
		 u(t,x) ^{-\alpha} - U_M(t, x) ^{-\alpha}
		 = \left[\left( \frac{ \rho-\rho_0  } { \alpha m_0(\rho_0) } \right)^{\frac {\alpha} {1-\alpha }} - \left( \frac{ \rho} { \alpha M} \right)^{\frac {\alpha} {1-\alpha }} \right] t^{-\frac \alpha {1-\alpha}}
	\end{equation}
	and hence
	\begin{align*}
		\frac{u(t,x) ^{\alpha} - U_M(t,x) ^{\alpha} }{U_M(t,x) ^{\alpha}} & =  \frac{ u( x  , t ) ^{-\alpha} - U_M(t,x) ^{-\alpha}}{u(t,x) ^{-\alpha}} \\
		&= \left[1 - \left( \frac{ \rho / ( \rho - \rho_0)  }  {  M / m_0(\rho_0) } \right)^{\frac {\alpha} {1-\alpha }}\right] \left[1+\alpha \left( \frac{ \rho-\rho_0  } { \alpha m_0(\rho_0) } \right)^{-\frac {\alpha} {1-\alpha }} t^{\frac 1 {1-\alpha}} \right]^{-1} . \qedhere
	\end{align*}
\end{proof}
Equation \eqref{formula} looks better in rescaled variables
\begin{align}
	\nonumber
	\left| \frac{w(t,y)^\alpha - F_M(|y|)^\alpha}{F_M(|y| )^\alpha} \right|
	&=   \left| \left[ 1 - \left( \frac{ |y|^d } { |y|^d - \rho_0(t,y) t^{-\frac 1 \alpha}   }  \frac { m_0(\rho_0(t,y)) }{M}  \right)^{\frac {\alpha} {1-\alpha }}\right]  \right| \\
	\nonumber
	&\quad \times  \left[1+\alpha \left( \frac{ |y|^d -\rho_0 (t,y) t^{-\frac 1 \alpha}  } { \alpha m_0(\rho_0 (t,y) ) } \right)^{-\frac {\alpha} {1-\alpha }}  \right]^{-1} \\
	\label{eq:relative comparison v and F power alpha}
	&\le  \left| 1 - \left( \frac{ |y|^d } { |y|^d - \rho_0(t,y) t^{-\frac 1 \alpha}   }  \frac { m_0(\rho_0(t,y)) }{M}  \right)^{\frac {\alpha} {1-\alpha }}\right|
\end{align}
where $\rho_0$ represents the foot of the rarefaction fan solution, i.e. in the notation of \Cref{ssect.rf}
\begin{equation*}
	\rho_0(t,y) = P_t^{-1} (t^{\frac 1 \alpha} y).
\end{equation*}
For compactly supported $u_0$, we have that $\rho_0 (t,y)$ is bounded and the first fraction tends to $1$ uniform in $t$. For the second fraction we need to now whether $m_0(\rho_0(t,y)) \to M$ when $t \to + \infty$.

\begin{lemma}
	\label{lem:rho_0 in rescaled variable}
	Let $u_0$ be radially non-increasing bounded and compactly supported. Let $\supp u_0=  [0,\rho_*]$. Then, for every $t \ge 0$, we have that
	\begin{equation*}
		\rho_0 (t,y) \to \rho_* \qquad \text{as } t \to +\infty \text{ uniformly for } |y| \ge \delta .
	\end{equation*}
\end{lemma}

\begin{proof}
	We write
	\begin{equation*}
		\rho = \rho_0 + \alpha \eta_0^{\alpha - 1} m_0 (\rho_0) t, \qquad \eta_0 \in [ u_0(\rho_0^+), u_0(\rho_0^-)].
	\end{equation*}
	In order to recover the scaling factor, we multiply by $t^{-\frac 1 \alpha}$ and bound some terms from below and some from above
	\begin{equation*}
		\omega_d \delta^d \le \omega_d |y|^d =   t^{-\frac 1 \alpha} \rho_0 + \alpha \eta_0^{\alpha - 1} m_0 (\rho_0) t^{-\frac{1 - \alpha}{\alpha}}
		\le t^{-\frac 1 \alpha} \rho_* + \alpha \eta_0^{\alpha - 1} M t^{-\frac{1 - \alpha}{\alpha}} .
	\end{equation*}
	Hence,
	\begin{equation*}
		\eta_0 \le t^{-\frac 1 \alpha} \left(  \frac {\alpha M  }  { \omega_d \delta ^d - t^{-\frac 1 \alpha} \rho_* }  \right) ^{ \frac 1 {1 - \alpha}}
	\end{equation*}
	so long as $t$ is large enough that $\omega_d y_0^d > t^{-\frac 1 \alpha} \rho_*  $. Since $\eta_0 \to 0$ uniformly in $|y_0|$ and $u_0$ is radially decreasing, then $\rho_0(t,y) \to \rho_*$ uniformly in $\delta$.
\end{proof}

We can now prove the main theorem
\begin{proof}[Proof of \Cref{thm:asymptotics t non-increasing u_0}]
	Let $\ee > 0$.
	Since $F_M$ is continuous and $F_M > 0$, let us take $\delta > 0 $ such that, if
	\begin{align}
		\label{eq:asymp delta from F_M}
		\left| \frac{F_M(|y|) - F_M(0)}{F_M(0)} \right| \le\ee  \qquad \forall |y| \le \delta .
	\end{align}
	
	\noindent \textbf{Step 1. Close to $y = 0$. One sided bound.} We assume first that $|y| < \delta$. Since $v$ is non-increasing in $y$, we have that
	\begin{equation*}
		w(t, y_0) \le w(t,y) \le w(t,0), \qquad \forall |y| \le |y_0| = \delta .
	\end{equation*}
	On the one hand we notice that
	\begin{equation*}
		w(t,0) = t^{\frac 1 \alpha} u(t,0) = ( u_0 t^{-1} + \alpha )^{-\frac 1 \alpha} \to   \alpha ^{-\frac 1 \alpha} = F_M(0)
	\end{equation*}
	as $t \to +\infty$. Hence, there exists $t_1 > 0$ such that
	\begin{equation*}
		w(t,0) \le F_M(0) (1 + \ee)  \qquad \forall t \ge t_1 .
	\end{equation*}
Hence,
	\begin{equation*}
		w(t,y) \le w(t,0) \le F_M(0) (1+\ee) \le F_M(|y|) (1 + \ee)^2,
	\end{equation*}
	Therefore,
	 \begin{equation}
	 	\label{eq:asymp t upper bound close 0}
	 	\frac{w(t,y) - F_M(|y|)}{F_M(|y|)} \le 2 \ee + \ee ^2  \qquad \forall t \ge t_1, |y| \le \delta.
	 \end{equation}
	
		\noindent \textbf{Step 2. Away from $0$. $|y| \ge \delta$.}
	Through \Cref{lem:relative error with self-similar power alpha} in version \eqref{eq:relative comparison v and F power alpha} and \Cref{lem:rho_0 in rescaled variable} we have that
	\begin{equation*}
		\sup_{|y| \ge \delta } \left| \frac{w(t,y)^\alpha - F_M(|y|)^\alpha}{F_M(|y| )^\alpha} \right|  \to 0, \qquad \text{as } t \to +\infty.
	\end{equation*}
	Therefore, there exists $t_2 > 0$ dependent on $\delta$ such that
	\begin{equation*}
		\frac{1}{2} F_M (|y|)^\alpha  \le w(t,y)^\alpha \le \frac 3 2 F_M(|y|)^\alpha , \qquad \forall t \ge t_2 , |y| \ge \delta.
	\end{equation*}
	Taking roots
	\begin{equation*}
		\frac{1}{2^{ \frac 1 \alpha  }} F_M (|y|)  \le w(t,y) \le   \left( \frac { 3 } { 2 } \right)^{{ \alpha}} F_M(|y|)
	\end{equation*}
	Since we want to compare $v$ and $F_M$ rather than their power, we use the intermediate value theorem that gives
	\begin{equation*}
		w(t,y)^\alpha - F_M(|y|)^\alpha = \alpha \nu (t,y) ^{\alpha - 1} (w(t,y) - F_M(|y|))
	\end{equation*}
	where $\nu(t,y)$ is between $w(t,y)$ and $F_M(|y|)$. Therefore
	\begin{align}
	\nonumber
	\left| \frac{ w(t,y)  - F_M(|y|)}{F_M(|y|)} \right|  &= \frac{1}{\alpha \nu(t,y)^{\alpha -1} } \left| \frac{ w(t,y) ^{\alpha} - F_M(|y|) ^{\alpha}}{F_M(|y|)} \right| \\
	\nonumber
	&=  \frac{1}{\alpha}\frac{ F_M(|y|) ^{\alpha - 1} }{ \nu(t,y)^{\alpha -1} } \left| \frac{ w(t,y) ^{\alpha} - F_M(|y|) ^{\alpha}}{F_M(|y|)^\alpha } \right| \\
	\nonumber
	&=  \frac{1}{\alpha} \left( \frac { \nu(t,y) } { F_M(|y|) } \right)^{1-\alpha} \left| \frac{ w(t,y) ^{\alpha} - F_M(|y|) ^{\alpha}}{F_M(|y|)^\alpha } \right| \\
	\label{eq:thm asymptotics t non-increasing u_0 comparison u^alpha and u}
	&\le \frac{1}{\alpha} \left( \frac { 3 } { 2 } \right)^{ \frac{ \alpha}{1-\alpha}} \left| \frac{ w(t,y) ^{\alpha} - F_M(|y|) ^{\alpha}}{F_M(|y|)^\alpha } \right|.
	\end{align}	
	Hence, exists $t_3 \ge t_2$ such that
	\begin{equation}
		\label{eq:asymptotics bound away}
		\left| \frac{ w(t,y)  - F_M(|y|)}{F_M(|y|)} \right| \le \ee, \qquad \forall t \ge t_3, |y|  \ge \delta.
	\end{equation}
	
	\noindent \textbf{Step 3. Close to $y = 0$. Other bound.} We write \eqref{eq:asymptotics bound away} as
	\begin{equation*}
		w(t,y) \ge (1 - \varepsilon) F_M(|y|), \qquad \forall t \ge t_3, |y| \ge \delta.
	\end{equation*}
	Therefore, taking some $|y_0| = \delta$ and using thtat $v$ in non-increasing in $y$ we have
	\begin{equation*}
		w(t,y) \ge w(t, y_0) \ge (1 - \varepsilon ) F_M(\delta), \qquad \forall t \ge t_3, |y| \le |y_0|.
	\end{equation*}
	Going back to \eqref{eq:asymp delta from F_M} we have that
	\begin{equation*}
		F_M(\delta) \ge (1-\varepsilon) F_M(0) \ge (1-\ee )^2 F_M(|y|).
	\end{equation*}
	so
	\begin{equation*}
		w(t,y) \ge F_M(|y|) (1 - \varepsilon)^3, \qquad \forall t \ge t_3, |y| \le \delta.
	\end{equation*}
	Therefore
	\begin{equation}
	\label{eq:asymp t lower bound close 0}
		\frac{w(t,y) - F_M(|y|)}{F_M(|y|)} \ge -3 \ee + 3 \ee^2 - \ee^3.
	\end{equation}
	Joining the information from \eqref{eq:asymp t upper bound close 0} and \eqref{eq:asymp t lower bound close 0} we have that
	\begin{equation*}
		\left| \frac{w(t,y) - F_M(|y|)}{F_M(|y|)} \right| \le 4 \ee + 4 \ee^2 + \ee^3, \qquad \forall t \ge t_2, t_3,\, |y| \le \delta.
	\end{equation*}
	Together with \eqref{eq:asymptotics bound away}, this completes the proof.
\end{proof}

\subsubsection{Asymptotic behaviour as $t \to +\infty$ for non-increasing data $u_0$ with an initial gap}
\label{sec:asymp in t non-increasing u_0 with gap}
Going back to what was said in \Cref{sec:data with initial gap}, if we have an initial datum
\begin{equation*}
u_0 (\rho) = \begin{dcases}
0 & \rho \le b , \\
\widetilde u_0(\rho - b) & \rho > b
\end{dcases}
\end{equation*}
where $u_0$ is non-increasing, then a solution of \eqref{eq:main equation} by characteristics is given by
\begin{equation*}
u (t,\rho) = \begin{dcases}
0 & \rho \le b , \\
\widetilde  u(t, \rho - b) & \rho > b
\end{dcases}
\end{equation*}
where $u$ is the solution by characteristics with datum $u_0$. In particular, $ u(t,0) = 0$ and, therefore, $ v (t,0) = 0$. Hence, \eqref{eq:asymp uniform convergence in relative error} cannot hold. Nevertheless, due to \Cref{thm:asymptotics t non-increasing u_0} we have the weaker form
\begin{equation}
\sup_{\omega_d |x|^d \ge    b } \left| \frac{  u(t,x)   -  U_M (t,x)  }{U_M \left ( t,x  \right )}\right|  \longrightarrow 0 ,  \qquad \textrm{as } t \to +\infty.
\end{equation}
In rescaled variable this reads
\begin{equation*}
\sup_{\omega_d |y|^d \ge  bt^{-\frac 1 \alpha} } \left| \frac{  w(t,y)   -  F_M(|y|)  }{ F_M(|y|) }\right|  \longrightarrow 0 ,  \qquad \textrm{as } t \to +\infty.
\end{equation*}
Asymptotically, this covers every $|y| > 0$. In particular, it guarantees that
\begin{equation}
	\sup_{|y|\ge \delta } \left| \frac{  w(t,y)   -  F_M \left (  {|y|}  \right ) }{F_M \left ( {|y|}  \right )}\right|  \longrightarrow 0 ,  \qquad \textrm{as } t \to +\infty, \quad \forall \delta > 0.
\end{equation}
This result is the most that can be expected in general.

\subsubsection{Asymptotic behaviour as $|x| \to + \infty$ for $t > 0$ fixed and general non-increasing data $u_0$}

When can repeat the argument to check that the tails of the self-similar solution are maintained as $|x| \to +\infty$ for any $t > 0$, if $u_0$ is compactly supported
\begin{theorem}
	\label{thm:non-decreasing analysis of the tails}
	Let $u_0 \in L^\infty_c (\mathbb R^d)$ radially non-increasing, $M = \| u_0 \|_{L^1}$, and let $u$ be given by \eqref{eq:solution radial nonincreasing discontinuous data}. Then, for every $t > 0$ fixed
	\begin{equation}
	\frac{u(t,x) - U_M(t,x)}{U_M(t,x)} \longrightarrow 0 , \qquad |x| \to +\infty.
	\end{equation}
\end{theorem}

\begin{proof}
	We repeat the same argument as before. First, from \Cref{lem:relative error with self-similar power alpha} and the fact that $\rho_0 = P_t^{-1} (\rho)$ we have that
	\begin{equation}
	\frac{u(t,x)^\alpha - U_M(t,x)^\alpha}{U_M(t,x)^\alpha} \longrightarrow 0 , \qquad |x| \to +\infty.
	\end{equation}
	Therefore, $u / U_M \in (c_1, c_2)$ at least for $|x|$ large. Then, arguing as in \eqref{eq:thm asymptotics t non-increasing u_0 comparison u^alpha and u} it holds that
	\begin{align*}
	\left| \frac{ u(t,x)  - U_M(t,x)}{U_M(t,x)} \right|
	&\le \frac{c_2 ^{ \frac{ \alpha}{1-\alpha}} }{\alpha} \left| \frac{ u(t,x) ^{\alpha} - U_M(t,x) ^{\alpha}}{U_M(t,x)^\alpha } \right| \to 0. \qedhere
	\end{align*}	
\end{proof}

\section{General radial data}
\label{sec.gendata}

As we pointed out in item \ref{it:u0 two triangle} of \Cref{rem:comments on characteristics} if the data $u_0$ is  not monotone, then characteristics can cross instantaneously. This leads to the appearance of a shock using the conservation law theory as we will discuss next.

\subsection{Shocks. The Rankine-Hugoniot equation}\label{sec:existhevanish}
	The choice of the free boundary is not trivial in principle. As usual in conservation laws, let us assume that  the solution is classical at either side of a shock wave $(t, S(t))$, and let us denote these solutions by $u^+$ and $u^-$.
	
	If we consider the mass at either side $S$ and continuity means that
	\begin{equation*}
		m(t,S(t)^+) = m(t,S(t)^-).
	\end{equation*}
Taking derivatives
	\begin{equation*}
		m(t, S(t)^+)_t + m_\rho(t, S(t)^+) S' = m_t(t, S(t)^-) + m_\rho (t, S(t)^-)S'.
	\end{equation*}
	Due to \eqref{eq:mass Burgers rho} and the continuity of $m$
	\begin{equation*}
		-m_\rho(t, S(t)^+)^\alpha m(t, S(t)) + m_\rho(t, S(t)^+) S' = -m_\rho(t, S(t)^-)^\alpha m(t, S(t)) + m_\rho(t, S(t)^-) S'.
	\end{equation*}
	Taking into account that $m_\rho = u$, we deduce that
	\begin{equation*}
		-(u^+)^\alpha m + u^+ S' = -(u^-)^\alpha m + u^-  S'
	\end{equation*}
	and solving for $S'$ we deduce the equation
	\begin{equation}
		\label{eq:RH}
		 S' (t) = m(t, S(t)) \frac{u^+(t,S(t))^\alpha - u^-(t,S(t))^\alpha}{u^+(t,S(t)) - u^-(t, S(t))} .
	\end{equation}
	We can call this the generalised Rankine-Hugoniot condition for \eqref{eq:mass Burgers rho}. We leave to the reader to check that if $u^+$ and $u^-$ are weak local solutions satisfying \eqref{eq:RH} then the solution constructed by pasting them
	\begin{equation*}
		u(t,x) = \begin{dcases}
			u^+(t, x) & x < S(t) , \\
			u^-(t,x) & x > S(t)
		\end{dcases}
	\end{equation*}
	is a weak local solution.

	\begin{remark}
		For solutions jumping at the boundary of the support formula \eqref{eq:RH} simplifies into
		\begin{equation}
			\label{eq:runkine support}
			S' (t) = m(t, S(t)) {u^+(t,S(t))^{\alpha-1}}.
		\end{equation}
	\end{remark}

	\subsection{Example of non-uniqueness of weak solutions: the square functions}\label{sec:spurious}

	Let $u_0 (\rho) = c_0 \chi_{[0,L]}(\rho)$. Then, the mass in volumetric coordinates is
	\begin{equation*}
		m(0,\rho) = \begin{dcases}
			c_0\rho & \rho < L , \\
			c_0 L & \rho \ge L .
		\end{dcases}
	\end{equation*}
	Let us find weak local solutions by characteristics. For $0 < \rho_0 < L$ we have that
	\begin{equation*}
		\rho(t) = \rho_0 + \alpha \rho_0 c_0^{\alpha} t .
	\end{equation*}
	We can therefore invert
	\begin{equation*}
		\rho_0 = \frac{r}{1 + \alpha c_0^{\alpha} t}.
	\end{equation*}
	Assume $0 < \rho/(1+\alpha c_0^{\alpha} t) < L$, then we can follow a characteristic back to $\rho_0 < L$.
	As a consequence we get
	\begin{align*}
		u(t,\rho) =
			c_0 (1 + \alpha c_0^{\alpha} t)^{-\frac 1 \alpha}.
	\end{align*}

	Using the generalised Rankine-Hugoniot condition \eqref{eq:RH},  we select two weak local solutions $\allowbreak u^+(t, \rho) = (c_0^{-\alpha} + \alpha t)^{-\frac 1 \alpha}$ and $u^- (t,\rho) = 0$. 		
		The mass $m(t,r)$ is clearly
		\begin{equation*}
			m(t,\rho) = (c_0^{-\alpha} + \alpha t)^{-\frac 1 \alpha} \Big( \rho \wedge S(t) \Big) .
		\end{equation*}
		Therefore, we substitute in \eqref{eq:runkine support} to deduce that
		\begin{equation*}
			S' (t) = S(t) (c_0^{-\alpha} + \alpha t)^{-\frac 1 \alpha} (c_0^{-\alpha} + \alpha t)^{-\frac {\alpha - 1} \alpha}.
		\end{equation*}
		Hence
		\begin{equation*}
			S' (t) = S(t) (c_0^{-\alpha} + \alpha t)^{-1}.
		\end{equation*}
		By definition the jump must begin at the jump of $u_0$, hence $S(0) = L$. We integrate the separable equation to deduce that
		\begin{equation*}
			S(t) = L (c_0^{-\alpha} + \alpha t)^{-\frac 1 \alpha}.
		\end{equation*}
		Therefore, the solutions obtained by elementary mass conservation arguments are precisely the weak solutions.
		We have therefore constructed the following function
		\begin{equation*}
			u(t,\rho) =
			\begin{dcases}
				c_0 (1 + \alpha c_0^{\alpha} t)^{-\frac 1 \alpha} & \rho < L (c_0^{-\alpha} + \alpha t)^{-\frac 1 \alpha}\\
				0 & \rho > L (c_0^{-\alpha} + \alpha t)^{-\frac 1 \alpha}.
			\end{dcases}
		\end{equation*}
		It is an additional weak local solution to the equation \eqref{eq:main equation} with initial data $u_0 = c_0 \chi_{[0,L]}$. Notice that we already constructed a rarefaction fan solution in \Cref{sec:rarefaction fan u_0 square}.
		We leave to the reader to check that this solution does not satisfy the Lax-Oleinik condition of incoming characteristics.		

\subsection{Initial data with two bumps}

We now consider the following initial data
\begin{equation*}
	u_0 = c_1 \chi_{[0,1]} + c_2 \chi_{[a,b]},
\end{equation*}
then physically meaningful solutions must develop a free boundary $S(t)$.
The solution can be computed classically on the left ($u_1(t,x)$) and right ($u_2(t,x)$) of $S(t)$. Presumably $S(0) = a$.
Since the characteristics emanating from the second bump have increasing slope, the Lax-Oleinik condition guarantees that $S$ is
increasing, so that they are incoming characteristics.
An ODE for $S(t)$ can be written via
the Rankine-Hugoniot equation \eqref{eq:RH}.

As $u_1$ we can take the solution by characteristics constructed in \Cref{sec:rarefaction fan u_0 square}, i.e.
\begin{equation*}
	u_1(t, \rho) =
	\begin{dcases}
	(c_1^{-\alpha} + \alpha t)^{-\frac 1 \alpha} & \rho \le 1 + \alpha c_1^{\alpha} t, \\
	\left( \left(\frac{\rho - 1}{\alpha c_1 t}\right)^{\frac \alpha {1-\alpha}} + \alpha t    \right)^{-1/\alpha} & \rho > 1 + \alpha c_1^{\alpha} t.
	\end{dcases}
\end{equation*}
For $u_2$ we have to do some additional computations. We will construct a solution by characteristics which is, as we have seen, a weak local solution. We look at the equation for the characteristics coming from $\rho_0 \in [a,b)$
\begin{equation*}
	\rho = \rho_0 + \alpha u_{2,0} (\rho_0)^{\alpha-1} m_{0,2} (\rho_0) t = \rho_0 + \alpha c_2^{\alpha-1} (c_1 + c_2 (\rho_0 - a)) t
\end{equation*}
where $m_{0,2}(a) = c_1$ is the mass accumulated from the first bump.
These characteristics correspond to the flat zone $u_2(t,\rho) = (c_2^{-\alpha} + \alpha t)^{-\frac 1 \alpha}$. On the other hand,
the characteristics of the rarefaction fan tail come from $\rho_0 = b$ and have as equations
\begin{equation*}
	\rho =  b + \alpha \eta_{0,2}^{\alpha-1} (c_1 + c_2 (b - a)) t
\end{equation*}
where each characteristic is determined by a unique value $\eta_{0,2} \in [0,c_2]$ and over this characteristic we have $u_2(t,\rho) = (\eta_{0,2}^{-\alpha} + \alpha t)^{-\frac 1 \alpha}$. Solving for $\eta_{0,2}$ we have
\begin{equation*}
	\eta_{0,2} = \left( \frac{\rho - b}{\alpha (c_1 + c_2 (b - a)) t } \right)^{\frac{1}{\alpha - 1}}.
\end{equation*}
Therefore, we can construct the solution right of the shock as
\begin{equation*}
	u_2(t,\rho) =
		\begin{dcases}
				\left( c_2^{- \alpha } + \alpha t    \right)^{-1/\alpha} & \rho \le b + \alpha (c_1 + c_2(b-a)) c_2^{\alpha-1} t\\
				\left( \left(\frac{\rho - b}{\alpha (c_1 + c_2 (b - a)) t }\right)^{\frac \alpha {1-\alpha}} + \alpha t    \right)^{-1/\alpha} & \rho > b + \alpha (c_1 + c_2(b-a)) c_2^{\alpha-1} t .
		\end{dcases}
\end{equation*}
We can compute the mass at the shock as that coming from the left. The shock will move faster than the characteristic coming from 1, and  Going back to \eqref{eq:mass characteristics}
\begin{align*}
	m(t, S(t)) &= m_1 (t, S(t)))
	= m_{0,1} (\rho_0) (1 + \alpha u_0( \rho_0 )^\alpha t)^{1- \frac 1 \alpha }\\
	&= \begin{dcases}
	c_1 S(t) (1 + \alpha c_1^\alpha t) ^{1-\frac 1 \alpha} &  S(t) \le 1 + \alpha c_1^{\alpha} t \\
	c_1 \left(1 + \alpha \left(  \frac{S(t) - 1}{{\alpha} c_1 t}  \right)^ {\frac \alpha {\alpha - 1}}  t\right)^{1-1/\alpha }, &  S(t) > 1 + \alpha c_1^{\alpha} t
	\end{dcases}
\end{align*}
and we end up with a piecewise defined ODE
\begin{equation*}
	S'(t) = m(t, S(t)) \frac{ u_1 (t, S(t))^\alpha -  u_2(t, S(t))^\alpha }{u_1(t, S(t)) - u_2(t, S(t))}.
\end{equation*}
The right hand side of the ODE for $S(t)$ is continuous and locally Lipschitz, so a unique solution exists. Solving numerically this ODE we obtain the characteristics given in \Cref{fig:two-bumps-characteristics-RK}.
\begin{figure}[H]
	\centering
	\includegraphics[width=.7\textwidth]{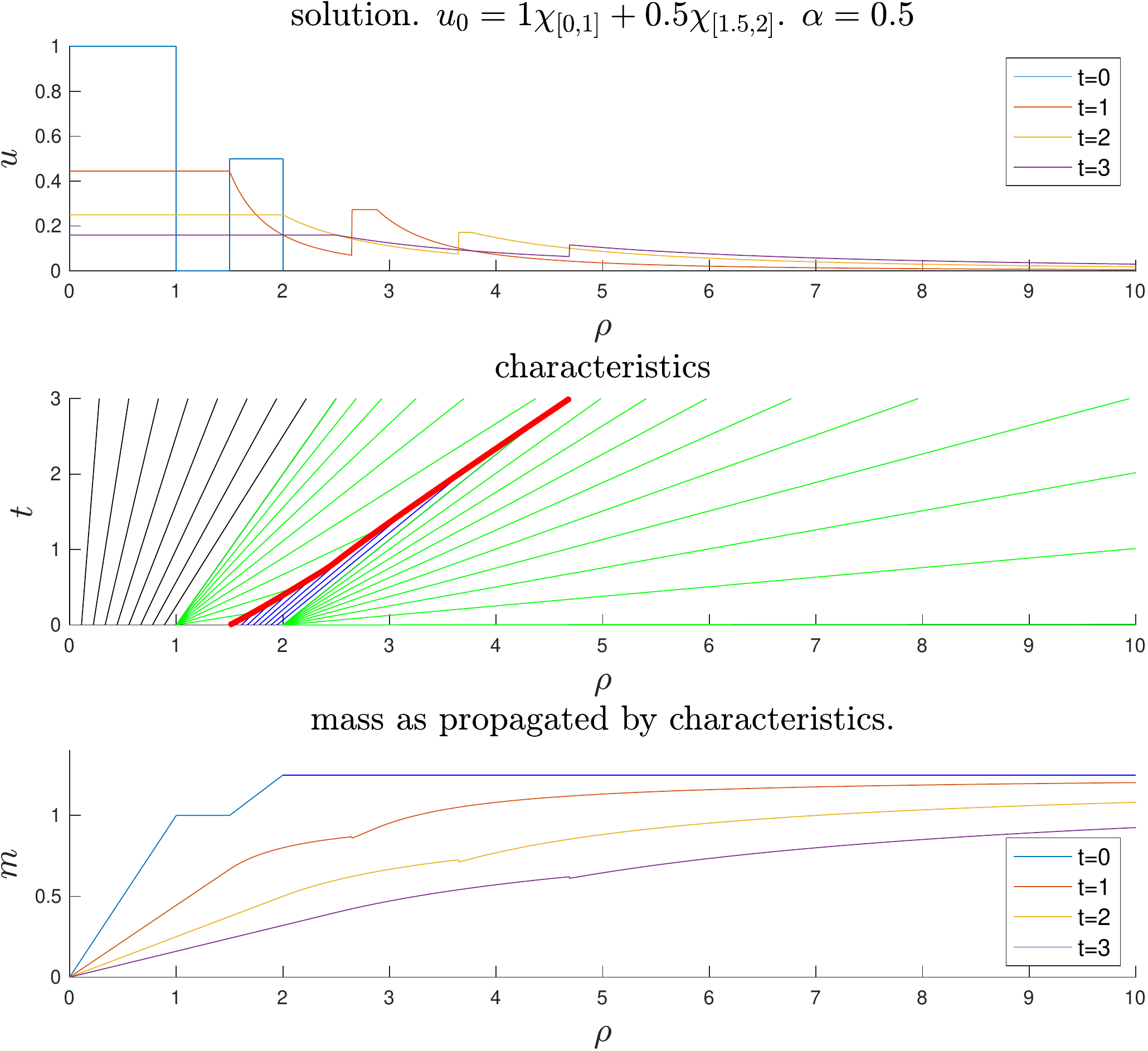}
	\caption{Solution using explicit solutions on either side of the shock, and a Runge-Kutta scheme to solve the Rankine-Hugoniot condition. We show only a few characteristics. The characteristics in black and blue correspond to the flat part of the solution, characteristics in green to the rarefaction fan tail, and the read line represents the shocks.}
	\label{fig:two-bumps-characteristics-RK}
\end{figure}

\section{Existence theory: vanishing viscosity}
\label{sec.exist}

Up to now we have constructed solutions by the method of characteristics. The problem is then to show that for this class of solutions the initial value problem is well-posed. As it is frequently done, we proceed further by constructing  first a well-posed theory for the approximate regularized problem \eqref{eq:PDE regularised}, and then we pass to the limit to construct solutions of \eqref{eq:main equation} that coincide with our previous constructions. Finally, we prove uniqueness of the limit by yet another theory.

\subsection{Classical solutions of the viscous problem \texorpdfstring{\eqref{eq:PDE regularised}}{(P eps)}}

We construct classical solution via fixed point of the heat equation

\begin{theorem}
	\label{thm:vanishing viscosity well-posedness}
	Let $u_0 \in L^1(\mathbb R^d) \cap \mathcal C_b(\mathbb R^d)$. Then, there exists a classical solution of \eqref{eq:PDE regularised}.
\end{theorem}
\begin{proof}
	We define the map
\begin{equation*}
	G: u \mapsto  \nabla \cdot ((\ee + u_+)^\alpha \nabla \N(u)).
\end{equation*}
One can write the equation as
$
	u_t - \ee \Delta u = G(u)
$.
Hence, it is natural to look for solutions as fixed points of Duhamel's formula
\begin{equation}
	\label{eq:Duhamel}
	u(t) = S_{\ee}(t) u_0 + \int_0^t S_{\ee} (t-s) G (u (s))  \diff s,
\end{equation}
where $S_\ee$ is the semigroup for $-\ee \Delta$. We recall the regularisation properties of the Newtonian potential
	$\N: \mathcal C_b(\mathbb R^d) \cap L^1 (\mathbb R^d) \cap L^2 (\mathbb R^d) \longrightarrow  \mathcal C^2_b(\mathbb R^d) \cap L^1 (\mathbb R^d) \cap  H^1 (\mathbb R^d) $.
By writing
\begin{equation*}
	G(u) = (\ee + u_+)^\alpha \nabla u \cdot \nabla N(u) - (\ee + u_+)^\alpha u,
\end{equation*}
we deduce that $
	G: H^1 (\mathbb R^d) \cap \mathcal C^1 (\mathbb R^d)  \longrightarrow  L^1(\mathbb R^d) \cap \mathcal C_b (\mathbb R^d).
$
Furthermore, it a Lipschitz operator. The operator from \eqref{eq:Duhamel} given by
\begin{equation*}
	K_t (u) = S_\ee(t) u_0 + \int_0^t S_\ee(t-s) G(u(s)) \diff s.
\end{equation*}
Due to the standard decay properties of the heat semigroup we have that
\begin{eqnarray*}
	K_t : X \longrightarrow  X, \qquad  \text{where } X = L^2 (0,T; L^1 (\mathbb R^d))\cap \mathcal C_b(\mathbb R^d)
\end{eqnarray*}
is Lipschitz with a constant depending on $t$. For $t$ small enough, the operator is contractive and we can use Banach's fixed point theorem to show there exists a unique solution of \eqref{eq:Duhamel}. Since $G$ is Lipschitz, this constant can be taken uniformly, and hence the solution is global. By a simple bootstrap argument, we show that the solution is classical.
\end{proof}

\begin{proposition}
	\label{prop:PDE regularised estimates}
	Let $u_0 \in L^1 (\mathbb R^d) \cap \mathcal C_b (\mathbb R^d)$. Then, any classical solutions of \eqref{eq:PDE regularised} satisfies
	\begin{equation}
		\| u(t,\cdot)_{+} \|_{L^p} \le \| (u_0)_{+} \|_{L^p} .
	\end{equation}
	The same holds for $u_-$, so if $u_0 \ge 0$, then $u \ge 0$.
\end{proposition}

\begin{proof}
	We study the positive and negative part separately. Studying the negative part of $u$ is very simple. For $1 < p < +\infty$, multiplying by $- (u_-)^{p-1}$ and integrating
	\begin{equation*}
		\int_{\mathbb R^d} u_t u_-^{p-1} - \int_{\mathbb R^d} (\ee + u_+)^\alpha \nabla N(u) \nabla u_-^{p-1} + \ee \int_{\mathbb R^d} |\nabla u_-|^2 = 0.	
	\end{equation*}
	Taking into account that $u$ is classical solution, that $(\ee + u_+)^\alpha \nabla u_- = \ee^\alpha \nabla u_-$ and the equation satisfied by $\N(u)$ we have that
	\begin{align*}
		\int_{\mathbb R^d} u_t u_-^{p-1} &=\frac{1}{p}\frac{\diff }{\diff t} \int_{\mathbb R^d} |u_-|^p \\
		- \int_{\mathbb R^d} (\ee + u_+)^\alpha \nabla N(u) \nabla u_-^{p-1} & = -\ee^\alpha \int_{\mathbb R^d} \nabla u_- \nabla \N(u) \\
		&= -\ee^\alpha \int_{\mathbb R^d} u u_- \ge  0.
	\end{align*}
	Hence, we can write
	\begin{equation*}
		\frac{\diff }{\diff t} \int_{\mathbb R^d} |u_-|^p \le  0.
	\end{equation*}
 	From this, we deduce that $\|u_-\|_{L^p}  \le \| (u_0)_-\|_{L^p}$.
	
	For $ 1 < p < +\infty$, we repeat a similar argument for $u_+$, multiplying by $u_+^{p-1}$ and integrating
	\begin{equation*}
		\frac{1}{p}\frac{\diff }{\diff t} \int_{\mathbb R^d} u_+^{p} + \int_{\mathbb R^d} (\ee+u_+)^\alpha \nabla N(u) \nabla u_+^{p-1} + \ee \int_{\mathbb R^d} |\nabla u_+|^2  = 0.
	\end{equation*}
	We have that $(\ee+u_+)^\alpha \nabla u_+^{p-1} = (p-1)(\ee+u_+)^\alpha u_+^{p-2} \nabla u_+ = g(u_+) \nabla u_+ = \nabla G(u_+)$ where $G$ is a primitive of $g$ such that $G(0) = 0$. We can write	
	Taking that $\N(u)$ is the solution $-\Delta \N(u) = u$ we have that
	\begin{equation*}
		 \int_{\mathbb R^d}  \nabla N(u) \nabla G(u_+) = \int_{\mathbb R^d} u G(u_+) = \int_{\mathbb R^d} u_+ G(u_+) \ge 0.
	\end{equation*}
	Finally
	\begin{equation*}
		\frac{1}{p}\frac{\diff }{\diff t} \int_{\mathbb R^d} u_+^{p} \le  0.	
	\end{equation*}
	For $p = 1, +\infty$, the estimates hold by passing to the limit.
\end{proof}

\subsection{An equation for the mass of \texorpdfstring{\eqref{eq:PDE regularised}}{(P eps)}}

Let us now compute
\begin{equation*}
	m_\ee(t,r) = \int_{B_r} u_\ee (t,x) \diff x.
\end{equation*}
As above, we write the equation in radial coordinates as
\begin{equation*}
	\frac{\partial u_\ee}{\partial t} =  r^{-(d-1)} \frac{\partial }{\partial r} \left(  r^{d-1} ( (u_\ee)_+ + \ee )^\alpha  \frac{\partial v}{\partial r}\right)
	+ \ee r^{-(d-1)} \frac{\partial }{\partial r} \left( r^{d-1}  \frac{\partial u_\ee }{\partial r}\right) .
\end{equation*}
Integrating over $B_r$ we have that
\begin{equation*}
	\frac{\partial m_\ee}{\partial t} =  d \omega_d  r^{d-1} ( (u_\ee)_+ + \ee )^\alpha  \frac{\partial v}{\partial r} + \ee d \omega_d  r^{d-1}  \frac{\partial u_\ee }{\partial r}.
\end{equation*}
Again $ - d \omega_d r^{d-1} \frac{\partial v}{\partial r} = m_\ee$ and
\begin{equation*}
	u_\ee =  \frac{1}{d \omega_d r^{d-1}} \frac{\partial m_\ee}{\partial r}
\end{equation*}
so we have that
\begin{equation*}
		\frac{\partial m_\ee}{\partial t} =   - \left(  \left( \frac{1}{d \omega_d r^{d-1}} \frac{\partial m_\ee}{\partial r} \right )_+ + \ee \right)^\alpha  m_\ee
		+ \ee d \omega_d r^{d-1}  \frac{\partial }{\partial r} \left(  \frac{1}{ d \omega_d r^{d-1} } \frac{\partial m_\ee}{\partial r} \right).
\end{equation*}
Considering the change in variable $\rho = \omega_d r^d$ we have that $\frac{\partial }{\partial r} = {d \omega_d r^{d-1}} \frac{\partial }{\partial \rho}$ and hence
\begin{equation*}
		\frac{\partial m_\ee}{\partial t} =   - \left(  \left( \frac{\partial m_\ee}{\partial \rho} \right )_+ + \ee \right)^\alpha  m_\ee
	+  \ee (d \omega_d r^{d-1})^{2}\frac{\partial }{\partial \rho} \left(   \frac{\partial m_\ee}{\partial \rho} \right).
\end{equation*}
Replacing the last $r$ by $\rho$
\begin{equation}
	\label{eq:PDE regularised mass}
		\frac{\partial m_\ee}{\partial t} = - \left(  \left( \frac{\partial m_\ee}{\partial \rho} \right )_+ + \ee \right)^\alpha  m_\ee
+  \ee (d \omega_d^{\frac 1 d} \rho^{\frac{d-1} d})^{2} \frac{\partial^2 m_\ee }{\partial \rho^2}.
\end{equation}
Therefore, the equation contains a term with degenerate viscosity for $d\geq 2$.

\subsection{Weak solutions of \texorpdfstring{\eqref{eq:main equation}}{(P)} via vanishing viscosity}

We can show existence of a  weak solution of \eqref{eq:main equation} by letting $\varepsilon \to 0$.

\begin{theorem}
	\label{thm:vanishing viscosity ee to 0}
	Let $u^{(n)}$ be a sequence of solutions of \eqref{eq:PDE regularised} with $\ee = \frac 1 n$. Then, there is a sub-sequence converging weakly in $L^1_{loc} \cap L^\infty_+$, and the limit is in $L^1 \cap L^\infty_+$.
\end{theorem}
\begin{proof}
	Since $\| u_\ee  (t) \|_{L^\infty ( \Rd)} \le \|u_0 \|_{L^\infty  ( \Rd)}$, there exists $u \in L^\infty  (\Rd)$ such that, up to a subsequence, $u_\ee \rightharpoonup u$ weak-$\star$ in $L^\infty(Q_T)$. Hence $u_n \rightharpoonup u$ weakly in $L^p_{loc} (\Rd)$. Furthermore, by the lower semicontinuity of the norm, we have that $u \in L^1 \cap L^\infty (\Rd)$. We finally use the compactness properties of the singular potentials, we deduce the local strong convergence of $\nabla\N(u_\ee)$ in $L^2_{loc} (\Rd)$.
\end{proof}

Instead of trying to prove uniqueness of weak solutions under additional conditions, we integrate and consider the mass function for radial initial data. In the next section, we show uniqueness of solutions in the sense of viscosity solutions for the mass variable.

\section{A theory of viscosity solutions for the mass equation}
\label{sec.viscos}

	As shown above, weak solutions are not in general unique. This is a common problem of conservation laws. In some cases, this difficulty is overcome by introducing the notion of entropy solutions (see, e.g. \cite{Kruzkov1970,Carrillo1999,Andreianov2000}). Such solutions are stable under passage to the limit and regularisation. They are understood as the ``physically meaningful solutions''. This notion works well for scalar laws, but authors have failed to extend it to systems, as is our case.
	
	In one dimension, the primitive of entropy solutions of conservation laws (or of radial solution) is a solution of a Hamilton-Jacobi equation. The corresponding notion with uniqueness is that on \emph{viscosity solutions} introduced by Crandall and Lions in \cite{Crandall1983} (a nice explanation on the link of entropy and viscosity solutions can be found \cite{Endal2018}). The nice properties are well-understood in our times (see \cite{crandall+evans+lions2006,crandall+ishii+lions1992users-guide-viscosity,Crandall1983}, for a nice introduction to this theory we point the reader to \cite{Tran2019}). Furthermore, viscosity solutions are approximated by Finite-Difference schemes (see \cite{Crandall1984}).
	
For the sake of clarity, let us recall here that vanishing viscosity solutions and viscosity solutions in the sense of Crandall-Lions are quite different concepts, though they often give the same class of solutions in practice. The latter concept will be used here below.

\subsection{Viscosity solutions}

The equation for the mass is written
\begin{equation*}
	m_t + (m_\rho)^\alpha m = 0
\end{equation*}
which is not problematic since we know that $m_\rho = u \ge 0$. To make a general theory it is better to write
\begin{equation}
	\label{eq:mass Burgers rho +}
	m_t + (m_\rho)_+^\alpha m = 0.
\end{equation}
Then, the Hamiltonian $H(z,p_1,p_2) = (p_2)^\alpha z$ is defined and non-decreasing everywhere.

\medskip

Letting $G (z,p_1,p_2,q) = p_1 + H(z,p_2)$ we have a monotone function. This recalls the theory in \cite{Carrillo2019}. We are not exactly in their setting, since our function is not \emph{weakly increasing}. The authors prove that viscosity solutions of this equation are non-decreasing in $\rho$ ($\rho$-m in their notation). We could apply their existence theory, but not the uniqueness one. Still, the solutions for the general case are continuous, but not necessarily uniformly continuous.

We introduce the definition of viscosity solution for our problem and some notation
\begin{definition}
		Let $f: \Omega \subset \mathbb R^m \to \mathbb R$. We define the Fréchet subdifferential and superdifferential
		\begin{align*}
			D^- u(x) &= \left \{ p \in \mathbb R^m : \liminf_{y \to x} \frac{u(y) - u(x) - p (y-x)}{|y-x|} \ge 0 \right\}\\
			D^+ u(x) &= \left \{ p \in \mathbb R^m : \limsup_{y \to x} \frac{u(y) - u(x) - p (y-x)}{|y-x|} \le 0 \right\}.
		\end{align*}
		
	\end{definition}
	
We recall the following result
\begin{theorem}
	Let $\Omega \subset \mathbb R^m$ be an open set and $f: \Omega \to \mathbb R$ be a continuous function. Then, $p \in D^+ f(x)$ if and only if there exists a function $\varphi \in C^1(I)$ such that $D\varphi (x) = p$ and $f - \varphi$ as a local maximum at $x$.
\end{theorem}
With the initial condition $(m_0)_\rho = u_0$ and the fact that $u_0 \in L_{loc}^1(\mathbb R^d)$ (so that the mass over the point $\{0\} = B_0$ is null), we consider the Cauchy problem
\begin{equation}
	\label{eq:mass equation}
	\begin{dcases}
		m_t + (m_\rho)_+^\alpha m = 0 & t > 0, \rho > 0 \\
		m(0,\rho) = m_0(\rho), \\
		m(t,0) = 0.
	\end{dcases}
\end{equation}
The natural setting is with $m_0$ Lipschitz (i.e.\ $m_\rho = u \in L^\infty$) and bounded (i.e.\ $u \in L^1$).

\begin{definition}
	\label{defn:viscosity sub, super and solution}
	We say that a continuous function $m \in \mathcal C([0,+\infty)^2)$ is a:
	\begin{itemize}
		\item[$\bullet$] viscosity subsolution of \eqref{eq:mass equation} if
	\begin{equation*}
		p_1 + (p_2)_+^\alpha m(t, \rho) \le  0, \qquad \forall (t, \rho) \in \mathbb R_+^2 \text{ and } (p_1,p_2) \in D^+ m(t, \rho).
	\end{equation*}
	and $m(0,\rho) \le m_0(\rho)$ and $m(t,0) \le  0$.

	\item[$\bullet$]  a viscosity supersolution of \eqref{eq:mass equation} if
	\begin{equation*}
		p_1 + (p_2)_+^\alpha m(t, \rho) \ge   0, \qquad \forall (t, \rho) \in \mathbb R_+^2 \text{ and } (p_1,p_2) \in D^- m(t, \rho).
	\end{equation*}
	and $m(0,\rho) \ge m_0(\rho)$ and $m(t,0) \ge 0$.
	
	\item[$\bullet$] a viscosity solution of \eqref{eq:mass equation} if it is both a sub and supersolution.
	\end{itemize}
	
	\begin{remark}
		The more general theory in \cite{Carrillo2019} allows for discontinuous sub and supersolution provided they are respectively lower and upper semicontinuous.
	\end{remark}

\end{definition}

\subsection{Comparison principle for the mass}

\begin{theorem}
	\label{thm:comparison principle m}
	Let $m$ and $M$ be uniformly continuous sub and supersolution of \eqref{eq:mass equation} in the sense of Definition~\ref{defn:viscosity sub, super and solution}. Then $m \le M$.
\end{theorem}

We apply an old idea by Crandall and Lions \cite{Crandall1984} of variable doubling. For its application we follow the scheme as presented in \cite[Theorem 1.18]{Tran2019} there written for $u_t + H(D_x u) = 0$ with suitable modifications.
\begin{proof}
	Assume, towards a contradiction that
	\begin{equation*}
		\sup_{(t,\rho) \in [0,+\infty)^2} (m(t,\rho) - M(t,\rho)) = \sigma > 0.
	\end{equation*}
	Since both functions are continuous, there exists $(t_1,\rho_1)$ such that $m(t_1,\rho_1) - M(t_1, \rho_1) > \frac {3\sigma}4$. Clearly, $t_1 , \rho_1 > 0$. Let us take $\varepsilon$ and $\lambda$ positive such that
	\begin{equation*}
		\varepsilon < \frac{\sigma}{16(\rho_1 + 1)}, \qquad \lambda < \frac{\sigma}{16(t_1 + 1)}.
	\end{equation*}
	With this choice, we have that
	\begin{equation*}
		2 \varepsilon \rho_1^2 + 2 \lambda t_1 < \frac{\sigma}{4}.
	\end{equation*}
	For this $\varepsilon$ and $\lambda$ fixed, let us construct the variable doubling function:
	\begin{equation*}
		\Phi(t,s,\rho,\xi) = m(t,\rho) - M(s,\xi) - \frac{|\rho-\xi|^2 + |s-t|^2}{\varepsilon^2} - \varepsilon( \rho^2 + \xi^2 ) - \lambda (s + t).
	\end{equation*}
	This function is continuous and bounded above, so it achieves a maximum at some point. Let us name this maximum depending on $\ee$, but not on $\lambda$:
	\begin{equation*}
		\Phi (t_\ee, s_\ee, \rho_\ee, s_\ee) \ge \Phi  (t_1,t_1,\rho_1,\rho_1) > \frac{3\sigma}{4} - 2 \ee \rho_1^2 - 2 \lambda t_1 > \frac{\sigma}{2}.
	\end{equation*}
	In particular, it holds that
	\begin{equation}
		\label{eq:estimate m - M at max of Phi ee}
		m(t_\ee, \rho_\ee) - M(s_\ee, \xi_\ee)  \ge \Phi(t_\ee,s_\ee, \rho_\ee, \xi_\ee) > \frac{\sigma}{2} .
	\end{equation}
	
	\noindent \textbf{Step 1. Variables collapse.} As $\Phi(t_\ee, s_\ee, \rho_\ee, \xi_\ee) \ge \Phi(0,0,0,0) = 0$, we have
	\begin{equation*}
		 \frac{|\rho_\ee-\xi_\ee|^2 + |s_\ee-t_\ee|^2}{\varepsilon^2} + \varepsilon( \rho_\ee^2 + \xi_\ee^2 ) + \lambda (s_\ee + t_\ee) \le m(t_\ee, \rho_\ee) - M(s_\ee, \xi_\ee) \le C.
	\end{equation*}
	Therefore, we obtain
	\begin{equation*}
		|\rho_\ee - \xi_\ee| + |t_\ee - s_\ee| \le C \varepsilon, \qquad \text{and} \qquad \rho_\ee + \xi_\ee \le \frac{C}{\sqrt \ee}.
	\end{equation*}
	This implies that, as $\varepsilon \to 0$, the variable doubling collapses to a single point.
	
	\medskip
	
	We can improve the first estimate using that $\Phi(t_\ee, s_\ee, \rho_\ee, \xi_\ee) \ge \Phi(t_\ee, t_\ee, \rho_\ee, \rho_\ee)$. This gives us
	\begin{align*}
		\frac{|\rho_\ee-\xi_\ee|^2 + |s_\ee-t_\ee|^2}{\varepsilon^2} &\le M(t_\ee, \rho_\ee) - M(s_\ee, \xi_\ee) + \varepsilon(\rho_\ee^2 - \xi_\ee^2) + \lambda(t_\ee - s_\ee) \\
		&\le M(t_\ee, \rho_\ee) - M(s_\ee, \xi_\ee) + \varepsilon \frac{C}{\sqrt \ee} C \ee + C \ee.
	\end{align*}
	Since $M$ is uniformly continuous, we have that
	\begin{equation*}
		\lim_{\ee \to 0}\frac{|\rho_\ee-\xi_\ee|^2 + |s_\ee-t_\ee|^2}{\varepsilon^2} = 0.
	\end{equation*}
	
	\noindent \textbf{Step 2. For $\ee > 0$ sufficiently small, the points are interior.} We show that there exists $\mu$ such that $t_\ee, s_\ee , \rho_\ee, \xi_\ee \ge \mu > 0$ for $\varepsilon>0$ small enough.  For this, since $m$ and $M$ are uniformly continuous
	\begin{align*}
		\frac{\sigma}{2} &<  m(t_\ee, \rho_\ee) - M(s_\ee, \xi_\ee) \\
		&=m(t_\ee, \rho_\ee) - m(0, \rho_\ee) + m(0, \rho_\ee) - M(0,\rho_\ee) \\
		&\qquad + M(0,\rho_\ee) - M(t_\ee, \rho_\ee) + M(t_\ee, \rho_\ee) - M(s_\ee, \xi_\ee) \\
		&\le \omega(t_\ee) + \omega( |\rho_\ee - \xi_\ee| + |t_\ee -s_\ee| ),
	\end{align*}
	where $\omega \ge 0$ is a modulus of continuity (the minimum of the moduli of continuity of $m$ and $M$),
	i.e.\ a continuous non-decreasing function such that $\lim_{r \to 0} \omega(r) = 0$. For $\varepsilon > 0$ such that
	\begin{equation*}
		\omega( |\rho_\ee - \xi_\ee| + |t_\ee -s_\ee| ) < \frac{\sigma}{4},
	\end{equation*}
	we have $\omega(t_\ee) > \frac{\sigma}{4}$. The reasoning is analogous for $s_\varepsilon$.
	For $\rho_\ee$ we can proceed much in the same manner
	\begin{align*}
		\frac{\sigma}{2} & < m(t_\ee, \rho_\ee) - M(s_\ee, \xi_\ee) \\
		&= m(t_\ee, \rho_\ee) - m(t_\ee, 0) + M(t_\ee,0) - M(t_\ee, \rho_\ee) \\
		&\qquad  + M(t_\ee, \rho_\ee) - M(s_\ee,\xi_\ee) \\
		&\le \omega(\rho_\ee) + \omega( |\rho_\ee - \xi_\ee| + |t_\ee -s_\ee| ).
	\end{align*}
	And analogously for $\xi_\ee$.
	
	\noindent \textbf{Step 3. Choosing viscosity test functions.} With the construction we have made, the function $(t,\rho) \mapsto \Phi(t,s_\ee, \rho, \xi_\ee)$ has a maximum at $(t_\ee, \rho_\ee)$. Thus, so does the function
	\begin{equation*}
		(t,\rho) \mapsto m(t,\rho) - \left( \frac{|\rho-\xi_\ee|^2 + |t- s_\ee|^2}{\varepsilon^2} + \ee \rho^2 + \lambda t  \right) = m(t, \rho) - \bar \varphi_\ee (t,\rho).
	\end{equation*}
	We must be careful to ensure that the test function has contact with $m$ at the right point $(t_\ee,\rho_\ee)$:
	\begin{equation*}
		\varphi_\ee (t, \rho) = \bar\varphi_\ee(t,\rho) + m(t_\ee, \rho_\ee) - \bar \varphi_\ee(t_\ee,\rho_\ee).
	\end{equation*}
	In fact, this is equivalent to
	\begin{equation*}
		D\varphi_\ee (t_\ee, \rho_\ee) = D\overline \varphi_\ee (t_\ee, \rho_\ee) \in D^+ m(t_\ee, \rho_\ee).
	\end{equation*}
	Since $m$ is a viscosity subsolution, we recover
	\begin{equation}
		\label{eq:comparison m viscosity}
		\frac{2 (t_\ee - s_\ee)}{\varepsilon^2} + \lambda + \left( \frac{2(\rho_\ee - \xi_\ee) }{\varepsilon^2} + 2 \ee \rho_\ee  \right)_+^\alpha m(t_\ee, \rho_\ee) \le 0.
	\end{equation}
	Analogously, the following function has a minimum at $(s_\ee, \xi_\ee)$:
	\begin{equation*}
		(s,\xi) \mapsto M(s,\xi) - \left(- \frac{|\rho_\ee-\xi|^2 + |s-t_\ee|^2}{\varepsilon^2} - \ee \xi^2 - \lambda s  \right) = M(s, \xi) - \bar \psi_\ee (s,\xi).
	\end{equation*}
	Again, the correct test function is
	\begin{equation*}
		\psi_\ee (s, \xi) = \bar\psi_\ee(s,\xi) + M(s_\ee, \xi_\ee) - \bar \psi_\ee(s_\ee,\xi_\ee).
	\end{equation*}
	Since $M$ is a viscosity supersolution, we recover
	\begin{equation}
		\label{eq:comparison M viscosity}
		\frac{2 (t_\ee - s_\ee)}{\varepsilon^2} - \lambda  + \left( \frac{2(\rho_\ee - \xi_\ee) }{\varepsilon^2} + 2 \ee \xi_\ee  \right)_+^\alpha M (s_\ee , \xi_\ee) \ge 0.
	\end{equation}
	
	\noindent \textbf{Step 4. A contradiction.} Taking the difference between \eqref{eq:comparison m viscosity} and \eqref{eq:comparison M viscosity}, we have that
	\begin{align*}
		2\lambda &\le \left( \frac{2(\rho_\ee - \xi_\ee) }{\varepsilon^2} + 2 \ee \xi_\ee  \right)_+^\alpha M_\ee (s_\ee, \xi_\ee)
		 - \left( \frac{2(\rho_\ee - \xi_\ee) }{\varepsilon^2} + 2 \ee \rho_\ee  \right)_+^\alpha m_\ee (t_\ee, \rho_\ee)\\
		 &= \left[\left( \frac{2(\rho_\ee - \xi_\ee) }{\varepsilon^2} + 2 \ee \xi_\ee  \right)_+^\alpha - \left( \frac{2(\rho_\ee - \xi_\ee) }{\varepsilon^2} + 2 \ee \rho_\ee  \right)_+^\alpha\right] M(s_\ee, \xi_\ee)\\
		 &\quad + \left( \frac{2(\rho_\ee - \xi_\ee) }{\varepsilon^2} + 2 \ee \rho_\ee  \right)_+^\alpha (M(s_\ee, \xi_\ee) - m(t_\ee, \rho_\ee))\\
		 &\le C \left| 2 \ee (\rho_\ee - \xi_\ee) \right|^\alpha   - \frac{\sigma}2 \left( \frac{2(\rho_\ee - \xi_\ee) }{\varepsilon^2} + 2 \ee \rho_\ee  \right)_+^\alpha \\
		 &\le C \left| 2 \ee (\rho_\ee - \xi_\ee) \right|^\alpha,
	\end{align*}
	due to the Hölder continuity of $\tau \mapsto \tau_+^\alpha$, the fact that $M$ is bounded and \eqref{eq:estimate m - M at max of Phi ee}. As $\ee \to 0$ recover the contradiction
	$0 < 2 \lambda \le 0$,
	and the proof is complete.
\end{proof}

\begin{remark}
	In the argument above, the uniform continuity condition plays a key role. Notice that it is possible that $\rho_\ee, \xi_\ee  \to +\infty$, and hence the continuity must be uniform to obtain the comparison estimate.
	
It was pointed out in \cite[Remark 4.2]{crandall+evans+lions2006} that the assumption of uniform continuity can be weakened (with minor modifications to the proof) to uniform continuity of $u_0, v_0$ and uniform convergence of $u(x, t) \to u_0(x)$ and $v(x, t) \to v_0(x)$ as $t \to 0$.
\end{remark}

\begin{remark}
	Notice that this proof can be extended to equations of the form $m_t + H(m_\rho) m = 0$ where $H \ge 0$ and uniformly continuous.
\end{remark}

\begin{remark}
	Notice that \cite[Theorem V.3]{Crandall1983} covers the case $\alpha \ge 1$, and furthermore gives information on the cone of dependence. Naturally, in our setting there is no cone of dependence.
\end{remark}

As a simple consequence of the comparison principle, we can take advantage of our explicit solutions for $u$ in \Cref{sec:explicit solutions}. The mass corresponding to the friendly giant should be a global supersolution. We compute the corresponding mass, which gives
\begin{equation*}
	M(t, \rho) =  ( \alpha t ) ^{-\frac 1 \alpha}\rho.
\end{equation*}
This is a classical solution of the equation
\begin{equation*}
	M_t + (M_\rho)_+^\alpha M =
	-\frac 1 \alpha\rho ( \alpha t ) ^{-\frac 1 \alpha - 1} + \left(( \alpha t ) ^{-\frac 1 \alpha-1} \right ) ^\alpha \left( \rho ( \alpha t ) ^{-\frac 1 \alpha} \right) = 0.
\end{equation*}
It is uniformly continuous for $ t > \mu$. We can apply the proof of the comparison principle to this very nice classical solution, and hence we deduce that
\begin{equation}
	\label{eq:global bound m}
	m(t,\rho) \le (\alpha t)^{-\frac 1 \alpha} \rho
\end{equation}
holds for all uniformly continuous viscosity solutions.
\begin{remark}
	This is known for Burger's as the universal or absolute supersolution.
\end{remark}

\begin{remark}
	Notice that this implies $m(t,0) \le 0$ for all $t \ge 0$. Since $m \equiv 0$ is also a solution, we check that, for all $m_0 \ge 0$, then the viscosity solution satisfies $m(t,0) = 0$ and $m(t,\rho)\ge 0$.
\end{remark}

\begin{remark}
	Formula \eqref{eq:global bound m} shows us that, for initial data $m_0 \ge 0$ and for all $\rho$ fixed, $m(t, \rho) \to 0$ as $t\to+\infty$, i.e.\ eventually all mass travels to infinity.
\end{remark}

\subsection{The mass of \texorpdfstring{\eqref{eq:PDE regularised}}{(P eps)} converges to the viscosity solution of \texorpdfstring{\eqref{eq:mass equation}}{(\ref{eq:mass equation})}}

\begin{theorem}
	\label{thm:converges mass PDE regularised}
	Let $d \ge 1$, $u_0 \in L^1 (\mathbb R^d) \cap \mathcal C_b(\mathbb R^d)$ radially symmetric. Let $u_\ee$ be the solution of \eqref{eq:PDE regularised}, $m_\ee$ its mass. Then $m_\ee \to m$ in $\mathcal C_{loc}([0,+\infty) \times [0,+\infty))$ where  $m$ is the viscosity solution of \eqref{eq:mass equation}.
	Furthemore, $m$ is Lipschitz continuous (in variable $\rho$).
\end{theorem}

\begin{proof}
	We first point out that $0\le m_\ee \le \| u_0 \|_{L^1 (\mathbb R^d)}$ and
	\begin{equation}\label{eq:viscosity wellposedness bound m_rho}
	(m_\ee)_\rho = u_\ee \in [0, \|u_0\|_{L^\infty (\mathbb R^d)}].
	\end{equation}
	Furthermore, we know that the solution is classical and $(m_\ee)_{\rho \rho}$ is also continuous.
	Hence, by the Ascolí-Arzelá theorem there is a convergent subsequence $m_\ee \to  m$ in $\mathcal C_{loc}([0,+\infty) \times [0,+\infty))$. Let us now check that $m$ is viscosity solution. We begin by showing it is a viscosity subsolution and, likewise, one proves it is a supersolution.  Fix $t,\rho > 0$, and let $(p_1,p_2) \in D^+ m(t, \rho)$. Again, we can construct a function $\varphi \in \mathcal C^2$ such that $m - \varphi$ has a local maximum at $(t,\rho)$ and $p_1 = \varphi_t(t,\rho), p_2=\varphi_\rho(t, \rho)$. Now we need to prove that the quantity
	\begin{equation*}
	\varphi_t(t,\rho) + (\varphi_\rho(t, \rho))_+^\alpha \, m(t, \rho)
	\end{equation*}
	is non-positive. For this, we go back to the viscosity equation. Since $m^\ee \to m$ and $m - \varphi$ has a maximum at $(t, \rho)$, by \cite[Lemma 1.8]{Tran2019} there exists a sub-sequence $\ee$ and $(t_\ee,\rho_\ee)$ of values such that $m^\ee - \varphi$ as a maximum at $(t_\ee,\rho_\ee)$ and, furthermore, $(t_\ee,\rho_\ee) \to (t, \rho)$ as $\ee \to 0$. We go back to \eqref{eq:PDE regularised mass}
	\begin{equation}
	\frac{\partial m_\ee}{\partial t} + \left(  \left( \frac{\partial m_\ee}{\partial \rho} \right )_+ + \ee \right)^\alpha  m_\ee
	 =  \ee C_d(\rho)\frac{ \partial^2 m_\ee}{\partial \rho^2}  .
	\end{equation}
	where $C_d(\rho)$ is defined and positive outside $\rho = 0$. For $\ee $ small enough $\rho_\ee > 0$ and, hence
	\begin{equation*}
	\frac{\partial m_\ee}{\partial t} (t_\ee, \rho_\ee) = \frac{\partial \varphi}{\partial t} (t_\ee, \rho_\ee), \qquad   \frac{\partial m_\ee}{\partial \rho} (t_\ee, \rho_\ee) = \frac{\partial \varphi}{\partial \rho} (t_\ee, \rho_\ee)
	\end{equation*}
	and
	\begin{equation*}
	\frac{\partial^2 m_\ee}{\partial \rho^2} (t_\ee, \rho_\ee) \le  \frac{\partial^2 \varphi}{\partial \rho^2} (t_\ee, \rho_\ee).
	\end{equation*}
	Hence,
	\begin{align*}
	\varphi_t(t_\ee,\rho_\ee) + \Big(\ee + (\varphi_\rho(t_\ee,\rho_\ee))_+\Big)^\alpha \, m_\ee (t_\ee,\rho_\ee) &\le  \ee C_d(\rho_\ee)\frac{ \partial^2 \varphi }{\partial \rho^2}(t_\ee,\rho_\ee) .
	\end{align*}
	As $\ee \to 0$ we have that
	\begin{equation*}
	\varphi_t(t,\rho) +  (\varphi_\rho(t,\rho))_+^\alpha \, m (t,\rho) \le 0.
	\end{equation*}
	Hence, we recover the sign
	\begin{equation*}
	p_1 +  (p_2)_+^\alpha \, m (t,\rho) \le 0.
	\end{equation*}
	Since this viscosity solution $m$ is unique by the comparison principle, the whole sequence $m_\ee$ converges.
\end{proof}

\subsection{Stability}\label{stability}

\begin{theorem}
	\label{thm:stability}
	Let $m_j$ be viscosity subsolutions (resp. supersolutions) of \eqref{eq:mass equation}, and assume that $m_j \to m$ uniformly over compacts as $j \to +\infty$. Then, $m$ is a viscosity subsolution (resp. supersolution) of \eqref{eq:mass equation}.
\end{theorem}
\begin{proof}
	Fix $t,\rho > 0$, and let $(p_1,p_2) \in D^+ m(t, \rho)$. We can construct a function $\varphi \in \mathcal C^1$ (see \cite[Theorem 1.4]{Tran2019}) such that $m - \varphi$ has a local maximum at $(t,\rho)$ and $p_1 = \varphi_t(t,\rho), p_2=\varphi_\rho(t, \rho)$. Now we need prove that the quantity
	\begin{equation*}
		\varphi_t(t,\rho) + (\varphi_\rho(t, \rho))_+^\alpha \, m(t, \rho)
	\end{equation*}
	Since $m_j \to m$ and $m - \varphi$ has a maximum at $(t_j, \rho)$, by \cite[Lemma 1.8]{Tran2019} there exists a sub-sequence $\ee$ and $(t_j,\rho_j)$ of values such that $m^\ee - \varphi$ as a maximum at $(t_j,\rho_j)$ and, furthermore, $(t_j,\rho_j) \to (t, \rho)$ as $j \to +\infty$. Then we have that
	\begin{equation*}
		\varphi_t (t_j,\rho_j) + (\varphi_\rho(t_j,\rho_j))_+^\alpha \, m_j (t_j,\rho_j) \le 0.
	\end{equation*}
	As $j \to +\infty$ we recover the definition of viscosity subsolution for $m$.
\end{proof}

We can also prove continuous dependence on the data, using this result. If $m_{0,j} \to m_0$, then the solutions converge uniformly.

\subsection{Well-posedness}

The mass associated to solutions of problem \eqref{eq:PDE regularised} are always Lipschitz continuous, but the comparison principle holds true for uniformly continuous. Let us show that, for $m_0 \in \BUC$ (i.e.\ bounded and uniformly continuous), there is a viscosity solution $u \in \BUC$.

\begin{theorem}
	\label{thm:viscosity mass well-posedness}
	Let $m_0 \in \BUC(\mathbb R^d)$ be non-decreasing such that $m_0(0) = 0$. Then, there exists a unique bounded and uniformly continuous viscosity solution. If $m_0$ is Lipschitz, then so is $m$.
\end{theorem}

\begin{proof} We first prove the case of $m_0 \in \mathcal C^1$ and Lipschitz, and then the general case.

\smallskip

\noindent \textbf{Step 1. $m_0$ of class $C^1$ and Lipschitz.}
We can construct initial data in 1D by taking $u_0 (\rho) =  (m_0)_\rho (|\rho|) \in   L^1 (\mathbb R)\cap \mathcal C_b(\mathbb R)$ and we apply \Cref{thm:converges mass PDE regularised}. %

\noindent\textbf{Step 2. Approximation arguments}

\smallskip

\textbf{Step 2a. $m_0\in \BUC (\mathbb R^d)$.}
Then, it can be approximated from above by a decreasing sequence of Lipschitz functions $\overline m_{0,k}$ and from below by an increasing sequence of functions $\underline m_{0,k}$. We construct $\underline m_{0,k}$ as follows: it can be taken as a piecewise approximation of $m_0 - \frac 1 k$ by piecewise constant function so that the uniform distance is less that $1/2k$. The procedure is analogous for $\overline m_{0,k}$.

 By the comparison principle for $m$, the corresponding solutions are ordered $\underline m_{k} \le \overline m_{k} $,  $\underline m_{k}$ is increasing and $\overline m_k$ decreasing. Due to Dini's theorem, the pointwise limits exists and the convergence is uniform over compacts
\begin{equation*}
	\underline m = \lim_k \underline m_k \le \lim_k \overline m_k = \overline m.
\end{equation*}
By \Cref{thm:stability}, these two functions are viscosity solutions, both corresponding to initial data $m_0$. Since they are the uniform limit of uniformly continuous functions, $\overline m, \underline m$ are uniformly continuous. By the comparison principle, $\overline m = \underline m$ and can simply call this function $m$.

\textbf{Step 2b. $m_0$ Lipschitz continuous.}
If $m_0$ is not only $\BUC$ but also Lipschitz continuous, then $m_{0,k}$ are uniformly Lipschitz continuous, and due to \eqref{eq:viscosity wellposedness bound m_rho} the functions $m_k$ are also uniformly Lipschitz continuous. Hence, $m$ is Lipschitz continuous.
\end{proof}

\begin{remark}
	If we repeat the approximation argument for $m_0$ continuous, we can repeat the argument and construct two continuous  viscosity solutions $\overline m$ and $\underline m$. Since the comparison principle \Cref{thm:comparison principle m} cannot be applied, we do not know if they are the same.
\end{remark}

\begin{remark}
	Notice that the rarefaction fan solution for $u_0 = c_0 \chi_{[0,L]}$, studied in \Cref{sec:rarefaction fan u_0 square}, is continuous $(0,+\infty)\times (0,+\infty)$, therefore $m$ is differentiable and thus a classical solution of \eqref{eq:mass equation}, so the unique viscosity solution. The other solution constructed in \Cref{sec:spurious} is therefore, shown as the spurious solution.
\end{remark}

\subsection{Asymptotics for the mass}

\begin{theorem}
	\label{thm:asymptotics mass}
	Let $m_0 \in \BUC ([0,+\infty))$  be non-decreasing with $m_0(0) = 0$ and such that $(m_0)_\rho$ is compactly supported. Let us denote by
$$
G_M(\kappa)=\int \limits _{\omega_d |y|^d \le \kappa } F_M (|y|) \diff y
$$
the mass function of the selfsimilar solution with total mass $M$.
Then, the unique viscosity solution satisfies for all $\kappa_0 \ge 0$ that
	\begin{equation}
	\label{eq:asymp mass}
	\sup_{\kappa \ge \kappa_0 } \left| \dfrac{  m (t, t^{\frac 1 \alpha} \kappa ) -  G_M (\kappa)} {G_M (\kappa)}  \right| \longrightarrow 0 ,
	\qquad \text{as } \ t \to +\infty.
	\end{equation}
\end{theorem}
\begin{proof}
Let us define $M = m_0(+\infty)$. Consider the supersolution with initial datum
\begin{equation*}
	\overline m_0 (\rho) = \begin{dcases}
	\| (m_0)_\rho \|_\infty \rho & \rho < M / 	\| (m_0)_\rho \|_\infty \\
	M & \rho > M / 	\| (m_0)_\rho \|_\infty.
	\end{dcases}
\end{equation*}
and the subsolution with initial datum, for $\delta$ such that $m_0(\delta ) = M$
\begin{equation}
	\underline m_0 (\rho) = \begin{dcases}
	0 & \rho < \delta  , \\
	M(\rho - \delta) & \delta < \rho < \delta + 1 \\
	M & \rho \ge \delta + 1.
	\end{dcases}
\end{equation}
Since $\overline u_0 = (\overline m_0)_\rho$ and $\underline u_0 = (\underline m_0)_\rho$ are square functions (and therefore non-increasing up to an initial gap),  we know from \Cref{sec:asymp in t non-increasing u_0} and \Cref{sec:asymp in t non-increasing u_0 with gap} that
\begin{equation*}
	\sup_{ \omega_d |x|^d \ge \delta } \left| \frac{   \underline u (t,x)   -  U_M \left (  t,x \right ) }{U_M \left ( t,x \right )}\right|  \longrightarrow 0 ,
	\qquad
	\sup_{ x \in \mathbb R^d } \left| \frac{  \overline u (t,x)   -  U_M \left ( t,x  \right ) }{U_M \left ( t,x  \right )}\right|  \longrightarrow 0 ,  \qquad \textrm{as } t \to +\infty.
\end{equation*}
Furthermore, since we know they are given by characteristics, we can apply \Cref{thm:comparison principle characteristics} to deduce that
\begin{equation*}
	\underline u \le \overline u.
\end{equation*}
Let $\ee > 0$ be fixed.
From the previous estimate we have that, for some $t_\ee > 0$
\begin{equation*}
	(1- \ee) U_M (t, x) \le \underline u(t,x) \le \overline u (t, x) \le (1+ \ee) U_M (t, x), \qquad \forall t \ge 0, \omega_d |x|^d \ge \delta
\end{equation*}
We will write in detail only upper bounds for $\underline u$, which is more delicate due to the presence of $\delta$.
\begin{equation*}
	 (1- \ee) U_M (t, x) \le \underline u (t, x) \le (1+ \ee) U_M (t, x), \qquad \forall t \ge t_\ee, |x| \ge \delta .
\end{equation*}
We define the mass of the self-similar solution
\begin{equation*}
m_M (t, \rho) = \int_{\omega_d |x|^d \le \rho } U_M (t,x) \diff x.
\end{equation*}
Integrating from $\rho$ to $+\infty$
\begin{equation*}
	 (1 - \varepsilon) ( M  -  m_M(t, \rho) ) \le M - \underline m (t, \rho) \le (1 + \varepsilon) ( M  -  m_M(t, \rho) ).
\end{equation*}
From this, we deduce that
\begin{equation*}
	  -\ee M + (1 + \ee ) m_M(t, \rho)  \le \underline m (t, \rho) \le   \ee M + (1 - \ee ) m_M(t, \rho) .
\end{equation*}
We have that
\begin{equation*}
	-\ee  \left(  \frac{M}{m_M(t, \rho)} - 1 \right) \le \frac{\underline m (t, \rho) - m_M(t, \rho)} {m_M(t, \rho)} \le   \ee \left(  \frac{M}{m_M(t, \rho)} - 1 \right), \qquad \forall t \ge t_\ee, \rho \ge \delta.
\end{equation*}
Notice that $M \ge m_M$ so $M/m_M - 1 \ge 0$. From here on, we write the estimate in absolute value. For $\rho$ fixed $m_M(t, \rho)\to 0$, so this estimate is not very nice on its own. However, we can pass to the scaling $y = x t^{-\frac{1}{\alpha d}}$ so it rescaled volume variable
\begin{equation*}
\kappa = \rho t^{-\frac 1 {\alpha } }.
\end{equation*}
Let us look the profile of $m_M$. Going back to the definition of $U_M$, we recover by integration that
\begin{equation*}
	 m_M (t, t^{\frac 1 \alpha} \kappa) = \int \limits _{\omega_d |y|^d \le \kappa} F_M (|y|) \diff y = G_M(\kappa).
\end{equation*}
Hence, in this rescaled variable we have that
\begin{equation*}
		\left| \dfrac{  \underline m (t, t^{\frac 1 \alpha} \kappa ) - G_M(\kappa)} {G_M(\kappa)} \right| \le  \ee \left( \frac{M}{G_M(\kappa)} - 1 \right) ,
		\qquad \forall t \ge t_\ee, \kappa \ge \delta t^{-\frac 1 {\alpha}} .
\end{equation*}
If we want a uniform bound, we cannot run all the range of $\kappa$, nevertheless we can fix a single $\kappa_0 > 0$ and, since $G_M (\kappa)$ is increasing $\kappa$, get
\begin{equation*}
	\left| \dfrac{ \underline m (t, t^{\frac 1 \alpha} \kappa ) - G_M(\kappa)} {G_M(\kappa)} \right| \le   \ee \left(  \frac{M}{G_M(\kappa_0)} - 1 \right) ,
\qquad \forall t \ge t_\ee, \kappa \ge \kappa_0 .
\end{equation*}
Proceeding analogously for $\overline m$ we recover the same bounds and hence we recover \eqref{eq:asymp mass}.
\end{proof}
\section{A finite-difference scheme for the mass}
\label{sec.fds}

In this section we return to the consideration of nonnegative solutions with positive nondecreasing mass function.
Since we know that the characteristics arrive from the left due to $m_\rho$ being positive, we can construct an upwind explicit scheme. We discretise the space and time variable by $	t_n = n h_t$, $\rho_j = j h_\rho$, and propose the scheme
\begin{equation}
	\frac{M_i^{k+1} - M_i^k}{h_t} + \left( \frac{M_{i}^k - M_{i-1}^k}{h_\rho} \right)^\alpha M_i^{k+1} = 0.
\end{equation}
Factoring out $M_i^{k+1}$, we get
\begin{equation}
	\left( 1 + h_t \left( \frac{M_{i}^k - M_{i-1}^k}{h_\rho} \right)^\alpha \right)M_{i}^{k+1} = M_i^k,
\end{equation}
hence, we deduce
\begin{equation}
	M_i^{k+1} = \frac{ M_i^k  }{1 + h_t \left( \frac{M_{i}^k - M_{i-1}^k}{h_\rho} \right)^\alpha}.
\end{equation}
Unfortunately, this scheme is not a monotone scheme for $\alpha < 1$, see \cite{Crandall1984}, due to the presence of the power, and hence it cannot be studied in the natural fashion proposed in \cite{Crandall1984}.

Nevertheless, we can still propose a monotone scheme, given by regularising the power $\alpha$. Including the initial and boundary conditions, one can write
\begin{equation}
	\tag{M$_\delta$}
	\label{eq:method delta}
	\begin{dcases}
			{M_j^{n+1}} = \frac{ M_j^n } { 1 +  h_t   H_\delta \left(\dfrac {M_j^n - M_{j-1}^n}{h_\rho} \right)  } & \text{if } j>0, n \ge 0 \\
			M_j^0 = m_0(h_\rho j) & \text{if } j \ge  0 \\
			M_0^n = 0 & \text{if }n > 0,
	\end{dcases}
\end{equation}
where
\begin{equation*}
	H_\delta (s) = (s_+ + \delta)^\alpha - \delta^\alpha.
\end{equation*}
Now, we can set a CFL condition such that the method is monotone. Indeed, if we assume that
\begin{equation}
		\tag{CFL$_\delta$}
		\label{eq:CFL delta}
		 \frac{h_t}{h_\rho} < \frac {\delta ^{1-\alpha }} { \alpha \overline M   },
\end{equation}
then the method is monotone if $M_j^n \in [0, \overline M]$.
Naturally, $\delta$ must go to zero as $h_t$ and $h_\rho$. It is easy to see that, if $m_0$ is non-decreasing in $\rho$ then so is $M_j^{n}$ in $j$.

\begin{theorem}
	\label{thm:convergence}
	Let $m_0$ be non-negative, non-decreasing, Lipschitz continuous and bounded, $M_j^n$ be the solution of \eqref{eq:method delta} and $m$ the viscosity solution of \eqref{eq:mass equation}. Let, for $\delta > 0$
	\begin{equation*}
		h_\rho = \delta^{{1 + 2\alpha}}  , \qquad h_t =\frac { \delta^{2+\alpha}} {2\alpha \| m_0 \|_\infty  }.
	\end{equation*}
	Then, for any $T > 0$
	\begin{equation*}
		\sup_{ \substack{ j \ge 0 \\ 0 \le n \le T / h_t}} |m(t_n, \rho_j) - M_j^n| \le C \delta^\alpha.
	\end{equation*}
	where $C = C(\alpha, T, \| m_0 \|_\infty, \|(m_0)_\rho\|_\infty)$. Hence, as $\delta \to 0$, the scheme (M$_\delta$) converges to the viscosity solution of \eqref{eq:mass equation}.
\end{theorem}

The lengthy details of the proof of this result are left to the interested reader following the blueprint of \cite{Crandall1984}. They crucially used the stability property of viscosity solutions proved in \Cref{stability}.

\begin{remark}
	We have performed a numerical simulation of the case with two bumps, in order to see if the results in \Cref{fig:two-bumps-characteristics-RK} are indeed the viscosity solution. The results can be seen in \Cref{fig:two-bumps-characteristics-FD}.
\end{remark}

\begin{figure}
	\centering
	\includegraphics[width=.7\textwidth]{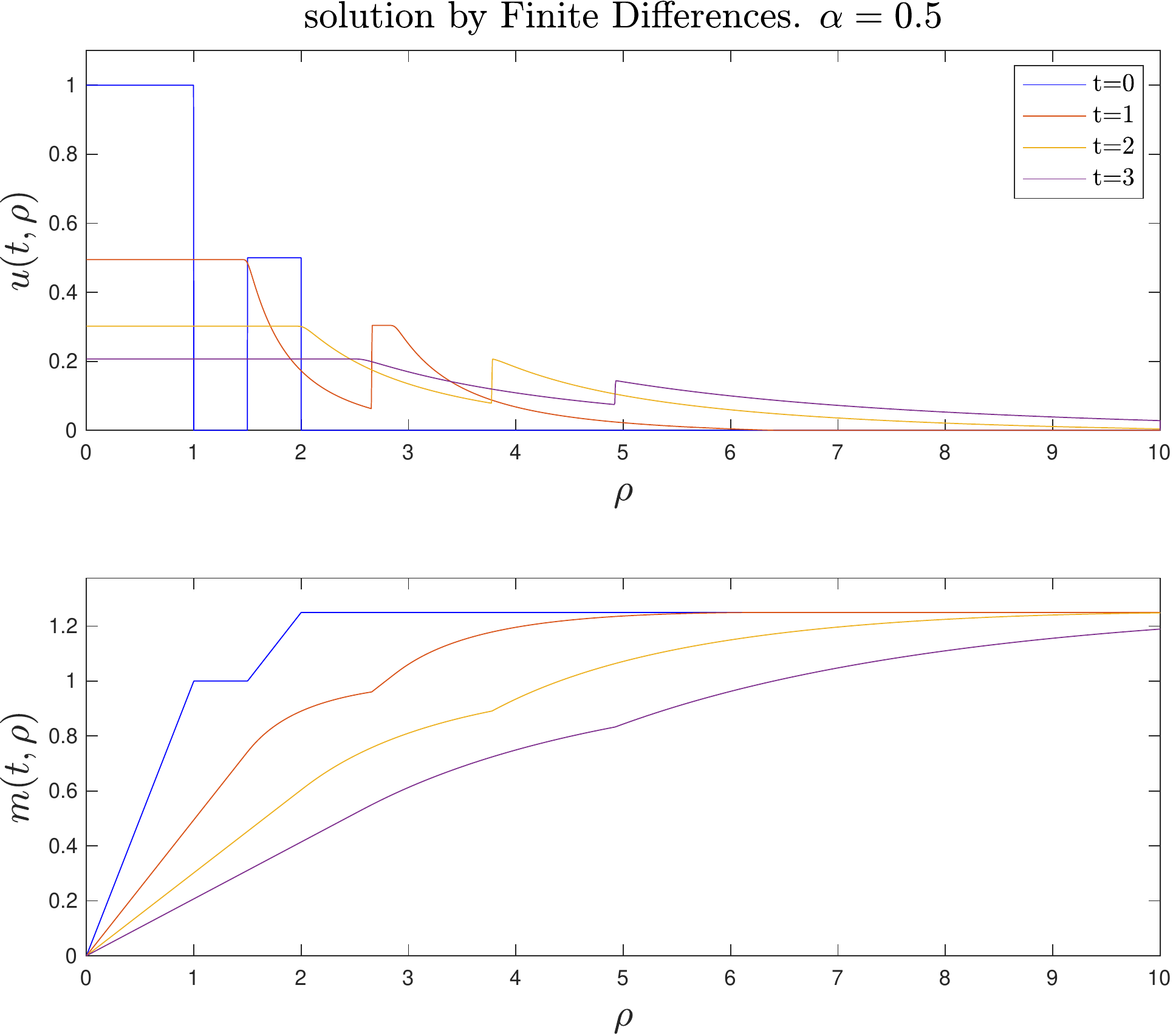}
	\caption{Solution by Finite Differences of the case with two bumps. Compare with \Cref{fig:two-bumps-characteristics-RK}.}
	\label{fig:two-bumps-characteristics-FD}
\end{figure}

\section{Extensions and open problems}
\label{sec.extop}
\begin{enumerate}
	\item The study of qualitative properties for non-radial solutions in several dimensions is completely open.
	\item The formal gradient flow structure of the equation may be relevant to discuss regularity and asymptotic properties of the equation for general initial data.	
	\item Another interesting question arises by looking at the attractive case or equivalently the backward evolution of our model. In the case $\alpha=1$, this was analyzed in \cite{BLL12} and is known in the literature as the skeleton problem.
\item Is there uniqueness of solutions for the mass equation with only continuous initial data?
	\item Can one construct convergent higher order numerical schemes for the mass equation?
\end{enumerate}

\medskip

\noindent {\bf Acknowledgments}.
JAC was partially supported by EPSRC grant number EP/P031587/1 and the Advanced Grant Nonlocal-CPD (Nonlocal PDEs for Complex Particle Dynamics:
Phase Transitions, Patterns and Synchronization) of the European Research Council Executive Agency (ERC) under the European Union’s Horizon 2020 research and innovation programme (grant agreement No. 883363).
The research of  DGC and JLV was partially supported by grant PGC2018-098440-B-I00 from the Ministerio de Ciencia, Innovación y Universidades of the Spanish Government.
JLV was an Honorary Professor at Univ.\ Complutense. DGC and JLV are grateful to Imperial College London, where most of this work was done while JLV held a Nelder fellowship.

\hyphenation{Sprin-ger Schatz-man equi-li-brium a-vaila-ble Mathe-ma-tics Me-cha-nics}

\printbibliography

\end{document}